\documentclass[a4paper, 11pt]{amsart}
\usepackage[utf8]{inputenc}
\usepackage[T1]{fontenc}
\usepackage[english]{babel}
\usepackage{graphicx}
\usepackage{stmaryrd}
\usepackage{tikz}
\usepackage{tikz-cd}
\usepackage[all]{xy}
\usepackage{comment}
\usepackage[a4paper]{geometry}
\usepackage{amsmath,amsfonts,amssymb,amsthm,epsfig,epstopdf,url,array}
\usepackage{rotating}
\usepackage{cleveref}
\usepackage{nameref}
\usepackage{enumitem}
\usepackage{comment}
\Crefname{paragraph}{Section}{Sections}
\usepackage{fancyhdr}

\newcommand{\ensemblenombre}[1]{\mathbb{#1}}
\newcommand{\N}{\ensemblenombre{N}}

\newcommand{\R}{} 
\renewcommand{\R}{\ensemblenombre{R}}
\newcommand{\C}{\ensemblenombre{C}}

\newcommand{\abs}[1]{\left\lvert#1\right\rvert}
\newcommand{\norme}[1]{\left\lVert#1\right\rVert}

\theoremstyle{plain} 
\newtheorem{prop}{Proposition}[section] 
\newtheorem{theo}[prop]{Theorem}
\newtheorem{defprop}[prop]{Definition-Proposition}
\newtheorem{lem}[prop]{Lemma}
\newtheorem{cor}[prop]{Corollary}
\theoremstyle{definition}
\newtheorem{defi}[prop]{Definition}
\newtheorem{rmk}[prop]{Remark}
\newtheorem{app}[prop]{Application}
\newtheorem{claim}[prop]{Fact}

\makeatletter
\let\original@addcontentsline\addcontentsline
\newcommand{\dummy@addcontentsline}[3]{}
\newcommand{\DeactivateToc}{\let\addcontentsline\dummy@addcontentsline}
\newcommand{\ActivateToc}{\let\addcontentsline\original@addcontentsline}
\makeatother

\begin{document}

\title{Null-controllability of two species reaction-diffusion system with nonlinear coupling: a new duality method}
\author{Kévin Le Balc'h}
\maketitle
\begin{abstract} 
We consider a $2\times2$ nonlinear reaction-diffusion system posed on a smooth bounded domain $\Omega$ of $\R^N$ ($N \geq 1$). The control input is in the source term of only one equation. It is localized in some arbitrary nonempty open subset $\omega$ of the domain $\Omega$. First, we prove a global null-controllability result when the coupling term in the second equation is an odd power.   As the linearized system around zero is not null-controllable, the usual strategy consists in using the return method, introduced by Jean-Michel Coron, or the method of power series expansions. In this paper, we give a direct nonlinear proof, which relies on a new duality method that we call Reflexive Uniqueness Method. It is a variation in reflexive Banach spaces of the well-known Hilbert Uniqueness Method, introduced by Jacques-Louis Lions. It is based on Carleman estimates in $L^p$ ($2 \leq p < \infty$) obtained from the usual Carleman inequality in $L^2$ and parabolic regularity arguments. This strategy enables us to find a control of the heat equation, which is an odd power of a regular function. Another advantage of the method is to produce small controls for small initial data. Secondly, thanks to the return method, we also prove a null-controllability result for more general nonlinear reaction-diffusion systems, where the coupling term in the second equation behaves as an odd power at zero.
\end{abstract}
\small
\tableofcontents
\normalsize
\section{Introduction}\label{intro}

Let $T > 0$, $N \in \N^*$, $\Omega$ be a bounded, connected, open subset of $\R^N$ of class $C^2$, and let $\omega$ be a nonempty open subset of $\Omega$. \\
\indent We consider a $2\times 2$ nonlinear reaction-diffusion system with one internal control:
\begin{equation*}
\left\{
\begin{array}{l l}
\partial_t u -  \Delta u =  f_1(u,v) + h 1_{\omega} &\mathrm{in}\ (0,T)\times\Omega,\\
\partial_t v -  \Delta v = f_2(u,v)  &\mathrm{in}\ (0,T)\times\Omega,\\
u,v= 0&\mathrm{on}\ (0,T)\times\partial\Omega,\\
(u,v)(0,.)=(u_0,v_0)& \mathrm{in}\ \Omega,
\end{array}
\right.
\tag{NL}
\label{systemef1f2}
\end{equation*}
with $f_1, f_2 \in C^{\infty}(\R^2;\R)$ satisfying $f_1(0,0)=f_2(0,0)=0$. Here,  $(u,v)(t,.) : \Omega \rightarrow \R^2$ is the \textit{state} to be controlled, $h = h(t,.) : \Omega \rightarrow \R$ is the \textit{control input} supported in $\omega$.\\
\indent We are interested in the \textbf{null-controllability} of \eqref{systemef1f2}: for any initial data $(u_0,v_0)$, does there exist a control $h$ such that the solution $(u,v)$ of \eqref{systemef1f2} verifies $(u,v)(T,.)=(0,0)$?

\subsection{Context}
The problem of null-controllability of the heat equation was solved independently by Gilles Lebeau, Luc Robbiano in $1995$ (see \cite{LR} or the survey \cite{LLR}) and Andrei Fursikov, Oleg Imanuvilov in $1996$ (see \cite{FI}) with Carleman estimates.
\begin{theo}\cite[Corollary 2]{AKBGBT}\\
For every $u_0 \in L^2(\Omega)$, there exists $h \in L^2((0,T)\times\Omega)$ such that the solution $u$ of 
\[
\left\{
\begin{array}{l l}
\partial_t u-  \Delta u =  h 1_{\omega} &\mathrm{in}\ (0,T)\times\Omega,\\
 u = 0 &\mathrm{on}\ (0,T)\times\partial\Omega,\\
 u(0,.)=u_0  &\mathrm{in}\ \Omega,
\end{array}
\right.
\]
satisfies $u(T,.)=0$.
\end{theo}
\indent Then, null-controllability of linear and nonlinear coupled parabolic systems has been a challenging issue. For example, in \cite{AKBDGB}, Farid Ammar-Khodja, Assia Benabdallah, Cédric Dupaix and Manuel Gonzalez-Burgos identified sharp conditions for the control of systems of the form
\begin{equation}
\left\{
\begin{array}{l l}
\partial_t U - D \Delta U = AU + B H 1_{\omega} &\mathrm{in}\ (0,T)\times\Omega,\\
U = 0&\mathrm{on}\ (0,T)\times\partial\Omega,\\
U(0,.)=U_0& \mathrm{in}\ \Omega.
\end{array}
\right.
\label{systemekalman}
\end{equation}
where $U(t,.): \Omega \rightarrow \R^n$ is the state, $H=H(t,.) : \Omega \rightarrow \R^m$ is the control, $D:=diag(d_1,\dots,d_n)$ with $d_i \in (0,+\infty)$ is the \textit{diffusion matrix}, $A \in \mathcal{M}_n(\R)$ (matrix with $n$ lines and $n$ columns with entries in $\R$) is the \textit{coupling matrix} and $B \in \mathcal{M}_{n,m}(\R)$ (matrix with $n$ lines and $m$ columns with entries in $\R$) represents the \textit{distribution of controls}.  In general, the rank of $B$ is less than $n$ (roughly speaking, there are less controls than equations), so that the controllability of the full system depends strongly on the (linear) coupling present in the system. We can see the survey \cite{AKBGBT} for other results (and open problems) on the controllability of linear coupled parabolic problems. The introduction of the article \cite{LB} provides an overview of the results on the controllability of linear and nonlinear coupled parabolic problems.\\
\indent Roughly speaking, the null-controllability of \eqref{systemef1f2} can be reformulated as follows: how can the component $v$ be controlled thanks to the nonlinear coupling $f_2(u,v)$?
\subsection{Linearization}
We introduce the following notation which will be used throughout the paper,
$$ \forall \tau > 0,\ Q_{\tau} = (0,\tau)\times\Omega.$$
The usual strategy consists in proving a \textbf{local null-controllability result} for \eqref{systemef1f2} from a \textbf{(global) null-controllability result for the linearized system of \eqref{systemef1f2}  around} $((\overline{u},\overline{v}),\overline{h})=((0,0),0)$. The linearized system \eqref{systemelf1f2} is 
\begin{equation*}
\left\{
\renewcommand*{\arraystretch}{1.5}
\begin{array}{l l}
 \partial_t u -  \Delta u =  \frac{\partial f_1}{\partial u}(0,0) u + \frac{\partial f_1}{\partial v}(0,0) v +h 1_{\omega} &\mathrm{in}\ (0,T)\times\Omega,\\
 \partial_t v -  \Delta v =  \frac{\partial f_2}{\partial u}(0,0) u + \frac{\partial f_2}{\partial v}(0,0) v  &\mathrm{in}\ (0,T)\times\Omega,\\
u,v= 0&\mathrm{on}\ (0,T)\times\partial\Omega,\\
(u,v)(0,.)=(u_0,v_0)& \mathrm{in}\ \Omega.
\end{array}
\right.
\tag{L}
\label{systemelf1f2}
\end{equation*}
\begin{defi}\label{definullcontr}
System \eqref{systemelf1f2} is said to be \textit{null-controllable} if \\for every $(u_0,v_0) \in L^2(\Omega)^2$, there exists $h \in L^2(Q_T)$ such that the solution $(u,v)$ of \eqref{systemelf1f2} satisfies $(u,v)(T,.) = (0,0)$.
\end{defi}  
\begin{prop}\cite[Theorem 7.1]{AKBGBT}\label{couplageconstant}\\
The following statements are equivalent.
\begin{enumerate}[nosep]
\item System \eqref{systemelf1f2} is null-controllable.
\item $ \frac{\partial f_2}{\partial u}(0,0) \neq 0$.
\end{enumerate}
\end{prop}
\indent Indeed, if $ \frac{\partial f_2}{\partial u}(0,0) = 0$, then the equation on $v$ is decoupled from the first equation of \eqref{systemelf1f2}. Consequently, for any initial data $(u_0,v_0) \in L^2(\Omega)^2$ such that $v_0 \neq 0$, we have $v(T,.)\neq 0$ by the backward uniqueness of the heat equation (see \cite{BT}). The proof of $(2) \Rightarrow (1)$ is a byproduct of \Cref{couplzonecontrole}.\\
\indent Roughly speaking, $u$ can be driven to $0$ thanks to the control $h$ and $v$ can be driven to $0$ thanks to the \textit{coupling term} $ \frac{\partial f_2}{\partial u}(0,0) u$. We have the following diagram
\[ h \overset{controls} \rightsquigarrow u \overset{controls} \rightsquigarrow v.\]
\begin{defi}\label{deflocalglobalcontr}[Null-controllability]
\begin{enumerate}[nosep]
\item System \eqref{systemef1f2} is said to be \textit{\textbf{locally} null-controllable} if there exists $\delta > 0$ such that for every ${(u_0,v_0) \in L^{\infty}(\Omega)^2}$ satisfying $\norme{(u_0,v_0)}_{L^{\infty}(\Omega)^2} \leq \delta$, there exists $h \in L^2(Q_T)$ such that \eqref{systemef1f2} has a (unique) solution $(u,v) \in L^{\infty}(Q_T)^2$ that satisfies $(u,v)(T,.) = (0,0)$.
\item System \eqref{systemef1f2} is said to be \textit{\textbf{globally} null-controllable} if\\ for every $(u_0,v_0) \in L^{\infty}(\Omega)^2$, there exists $h \in L^2(Q_T)$ such that \eqref{systemef1f2} has a (unique) solution $(u,v) \in L^{\infty}(Q_T)^2$ that satisfies $(u,v)(T,.) = (0,0)$.
\end{enumerate}
\end{defi}
\indent Now, we mention the \textbf{linear test} for \eqref{systemef1f2} which is a corollary of \Cref{couplageconstant}.
\begin{prop}\cite[Proof of Theorem 1]{CGR}\label{nonlineartoy}\\
Let us suppose that $\frac{\partial f_2}{\partial u}(0,0) \neq 0$. Then, \eqref{systemef1f2} is locally null-controllable.
\end{prop}
\begin{rmk}
This result is well-known but it is difficult to find in the literature (see \cite[Theorem 6]{AKBD} with a restriction on the dimension $1 \leq N < 6$ and other function spaces or one can adapt the arguments given in \cite{CGR} to get \Cref{nonlineartoy} for any $N\in \N^*$). For other results in this direction, see \cite{WZ}, \cite{LCMML}, \cite{GBPG} and \cite{CSG}.
\end{rmk}
The natural question is: what can we say about \eqref{systemef1f2} if the \textbf{linearized system around $((0,0),0)$ is not null-controllable} i.e. when $ \frac{\partial f_2}{\partial u}(0,0) = 0$?\\
\indent Another strategy to get local null-controllability for \eqref{systemef1f2} consists in \textbf{linearizing around a non trivial trajectory} $(\overline{u},\overline{v},\overline{h}) \in C^{\infty}(\overline{Q_T})^3$ of the nonlinear system  \eqref{systemef1f2}  which goes from $0$ to $0$. This procedure is called the \textbf{return method} and was introduced by Jean-Michel Coron in \cite{C2} (see \cite[Chapter 6]{C}). The linearized system is the following one:
\begin{equation*}
\left\{
\renewcommand*{\arraystretch}{1.5}
\begin{array}{l l}
\partial_t u -  \Delta u =  \frac{\partial f_1}{\partial u}(\overline{u},\overline{v}) u + \frac{\partial f_1}{\partial v}(\overline{u},\overline{v}) v +h 1_{\omega} &\mathrm{in}\ (0,T)\times\Omega,\\
\partial_t v -  \Delta v =  \frac{\partial f_2}{\partial u}(\overline{u},\overline{v}) u + \frac{\partial f_2}{\partial v}(\overline{u},\overline{v}) v  &\mathrm{in}\ (0,T)\times\Omega,\\
u,v= 0&\mathrm{on}\ (0,T)\times\partial\Omega,\\
(u,v)(0,.)=(u_0,v_0)& \mathrm{in}\ \Omega.
\end{array}
\right.
\tag{L-bis}
\label{systemelrmf1f2}
\end{equation*}
First, let us recall the generalization of \Cref{couplageconstant} when the \textit{coupling coefficients are not constant}. Historically, the proof is due to Luz de Teresa in \cite{LdT}.
\begin{prop}\cite[Theorem 7.1]{AKBGBT}\label{couplzonecontrole}\\
We assume that there exist $t_1 < t_2 \in (0,T)$, a nonempty open subset $\omega_0 \subset \omega$ and $\varepsilon > 0$ such that $\left|\frac{\partial f_2}{\partial u}(\overline{u}(t,x),\overline{v}(t,x))\right| \geq \varepsilon$ for every $ (t,x) \in (t_1,t_2)\times \omega_0$. Then, system \eqref{systemelrmf1f2} is null-controllable (in the sense of \Cref{definullcontr}).
\end{prop}
Then, the linear test gives the following result.
\begin{prop}\cite[Proof of Theorem 1]{CGR}\label{nonlineartoy2}\\
We assume that there exist $t_1 < t_2 \in (0,T)$, a nonempty open subset $\omega_0 \subset \omega$ and $\varepsilon > 0$ such that $\left|\frac{\partial f_2}{\partial u}(\overline{u},\overline{v})\right| \geq \varepsilon$  on $(t_1,t_2) \times \omega_0$. Then, system \eqref{systemef1f2} is locally null-controllable.
\end{prop}
\Cref{nonlineartoy2} is used used in \cite{CGR} with $f_2(u_1,u_2)= u _1^3 + R u_2$, where $R \in \R$, \cite{CGMR}, \cite{CG} and \cite{LB}.
\subsection{The “power system”}
\indent A model-system for the question of null-controllability when the linearized system around $((0,0),0)$ is not null-controllable is the following one:
\begin{equation}
\left\{
\begin{array}{l l}
\partial_t u -  \Delta u =   h 1_{\omega} &\mathrm{in}\ (0,T)\times\Omega,\\
\partial_t v -  \Delta v = u^{n}  &\mathrm{in}\ (0,T)\times\Omega,\\
u,v= 0&\mathrm{on}\ (0,T)\times\partial\Omega,\\
(u,v)(0,.)=(u_0,v_0)& \mathrm{in}\ \Omega,
\end{array}
\right.
\tag{Power}
\label{systemeu^n}
\end{equation}
where $n \geq 2$ is an integer.
\begin{prop}\label{resultatnegatifeven}
If $n$ is an even integer, then \eqref{systemeu^n} is not locally null-controllable.
\end{prop}
Indeed, by the \textit{maximum principle} ($u^n \geq 0$), we have, for any solution of \eqref{systemeu^n} associated to an initial condition $(u_0, v_0)$ with $v_0 \geq 0$ and $v_0 \neq 0$,
\[ v(T,.) \geq \widetilde{v}(T,.) \geq 0\qquad \text{and}\qquad \widetilde{v}(T,.) \neq 0,\]
where $\widetilde{v}$ is the solution of the heat equation
\begin{equation*}
\left\{
\begin{array}{l l}
\partial_t \widetilde{v}-  \Delta \widetilde{v} = 0  &\mathrm{in}\ (0,T)\times\Omega,\\
\widetilde{v}= 0&\mathrm{on}\ (0,T)\times\partial\Omega,\\
\widetilde{v}(0,.)=v_0& \mathrm{in}\ \Omega.
\end{array}
\right.
\label{heatequationmax}
\end{equation*}
\indent The following result is due to Jean-Michel Coron, Sergio Guerrero and Lionel Rosier. The proof is based on the return method (see \cite{CGR}).
\begin{prop}\cite[Theorem 1]{CGR}\label{Powern=3}\\
If $n=3$, then \eqref{systemeu^n} is locally null-controllable.
\end{prop}
\begin{rmk}
The difficult point of the proof of \Cref{Powern=3} is the construction of the nontrivial trajectory (see \cite[Section 2]{CGR}). The method can be generalized to $n=2k+1$ for $k \in \N^*$ but with longer computations. The same problem appears in \cite[Section 4.2]{Zh}.
\end{rmk}
\begin{rmk}
An homogeneity argument shows that for the system \eqref{systemeu^n} the local null-controllability implies the global null-controllability (consider $u_{\varepsilon} = \varepsilon u$, $v_{\varepsilon} = \varepsilon^n v$, $h_{\varepsilon} = \varepsilon h$). However, this strategy does not provide estimate on the control. This kind of argument is used in \cite{CG}. In this paper, we propose a different direct method for the global null-controllability, that provides estimates.
\end{rmk}

\subsection{A direct approach}
From now on, $k\in \N^*$ is fixed.\\
\indent The first goal of this paper is to give \textit{a direct proof (i.e. without return method)} of the \textit{global null-controllability} of the system 
\begin{equation*}
\left\{
\begin{array}{l l}
\partial_t u -  \Delta u =   h 1_{\omega} &\mathrm{in}\ (0,T)\times\Omega,\\
\partial_t v -  \Delta v = u^{2k+1}  &\mathrm{in}\ (0,T)\times\Omega,\\
u,v= 0&\mathrm{on}\ (0,T)\times\partial\Omega,\\
(u,v)(0,.)=(u_0,v_0)& \mathrm{in}\ \Omega.
\end{array}
\right.
\tag{Odd}
\label{systemeodd}
\end{equation*}
Our proof is based on a \textit{new duality method}, called \textbf{Reflexive Uniqueness Method}. The first step will consist in proving a \textit{Carleman estimate in $L^{2k+2}$ for the heat equation} (see \Cref{carlestimate2k+2subsection} and particularly \Cref{carl2+2cor}). The second step will consist in considering a \textit{penalized problem in $L^{\frac{2k+2}{2k+1}}$} (see \Cref{RUM}), a generalization of the Penalized \textbf{Hilbert Uniqueness Method}, introduced by Jacques-Louis Lions (see \cite{LJL} and also \cite[Section 2]{Z-hand} for an introduction to the Hilbert Uniqueness Method and some generalizations). This procedure enables us to find \textit{a control of the heat equation which is an odd power of a regular function}.\\
\indent The second goal of this paper is to prove a \textit{local null-controllability} result for more general systems than \eqref{systemeodd}  thanks to the \textbf{return method} (introduced in \Cref{intro}).

\section{Main results}

\subsection{Definitions and usual properties}

\begin{defi}\label{defiYspaceL2}
We introduce the functional space 
\begin{equation}
W_T := L^2(0,T;H_0^1(\Omega)) \cap H^1(0,T;H^{-1}(\Omega)).
\end{equation}
\end{defi}

\begin{prop}\label{injclassique}\cite[Section 5.9.2, Theorem 3]{E}\\
We have the embedding
\begin{equation}
W_T \hookrightarrow C([0,T];L^2(\Omega)).
\end{equation}
\end{prop}
We have this well-posedness result for linear parabolic systems.
\begin{defprop}\label{wpl2}
Let $l \in \N^*$, $y_0 \in L^2(\Omega)^l$, $g \in L^{2}(Q_T)^l$. The following Cauchy problem admits a unique weak solution $y \in W_T^l $
\[
\left\{
\begin{array}{l l}
\partial_t y -  \Delta y=  g&\mathrm{in}\ (0,T)\times\Omega,\\
y = 0 &\mathrm{on}\ (0,T)\times\partial\Omega,\\
y(0,.)=y_0 &\mathrm{in}\  \Omega.
\end{array}
\right.
\]
This means that $y$ is the unique function in $W_T^l$ which satisfies the variational formulation
\begin{equation}
\forall w \in L^2(0,T;H_0^1(\Omega)^l),\ \int_0^T (\partial_t y ,w)_{(H^{-1}(\Omega)^l),H_0^1(\Omega)^l)} + \int_{Q_T}  \nabla y .  \nabla w = \int_{Q_T} g . w,
\label{formvar}
\end{equation}
and
\begin{equation}
y(0,.) = y_0 \ \mathrm{in}\ L^2(\Omega)^l.
\label{condinitl2}
\end{equation}
Moreover, there exists $ C >0$ independent of $y_0$ and $g$ such that
\begin{equation}
 \norme{y}_{W_T^l} \leq C \left(\norme{y_0}_{L^{2}(\Omega)^l}+\norme{g}_{L^{2}(Q_T)^l}\right).
 \label{estl2faible}
\end{equation}
Let $p \in [1,+\infty]$, $y_0 \in L^{p}(\Omega)^l$ and $g \in L^{p}(Q_T)^l$, then $y \in L^{p}(Q_T)^l$ and there exists $ C >0$ independent of $y_0$ and $g$ such that
\begin{equation}
\norme{y}_{L^{p}(Q_T)^l} \leq C \left(\norme{y_0}_{L^{p}(\Omega)^l}+\norme{g}_{L^{p}(Q_T)^l}\right).
\label{estlinftyfaible}
\end{equation}
\end{defprop}
\begin{proof}
First, the well-posedness in $W_T^l$ (i.e. \eqref{formvar}, \eqref{condinitl2} and \eqref{estl2faible}) is based on \textit{Galerkin approximations and energy estimates}. One can easily adapt the arguments given in \cite[Section 7.1.2]{E}.\\
\indent Secondly, the $L^p$-estimate (i.e. \eqref{estlinftyfaible} for $p<+\infty$) is based on the application of \eqref{formvar} with a cut-off of $w = |y|^{p-2}y$.\\
\indent Finally, the $L^{\infty}$-estimate (i.e. \eqref{estlinftyfaible} for $p=+\infty$) is based on \textit{Stampacchia's method} (see the proof of \cite[Chapter 3, Paragraph 7, Theorem 7.1]{LSU}).
\end{proof}
The following definition-proposition justifies the notion of (unique) solution associated to a control and therefore, the definition of local null-controllability and global null-controllability (already introduced in \Cref{intro}, see \Cref{deflocalglobalcontr}).
\begin{defprop}\label{defpropsolNL}
Let $(u_0,v_0) \in L^{\infty}(\Omega)^2$, $h \in L^2(Q_T)$.\\
\indent Let $(u,v) \in (W_T \cap L^{\infty}(Q_T))^2$. We say that $(u,v)$ is a solution of \eqref{systemef1f2} if $(u,v)$ satisfies 
\begin{align}
&\forall (w_1,w_2) \in L^2(0,T;H_0^1(\Omega))^2\notag,\\
&\label{formvarNL1}  \int_0^T (\partial_t u ,w_1)_{(H^{-1}(\Omega),H_0^1(\Omega))} + \int_{Q_T}  \nabla u .  \nabla w_1 = \int_{Q_T} (f_1(u,v)+h1_{\omega}) w_1,\\
&\label{formvarNL2}\int_0^T (\partial_t v ,w_2)_{(H^{-1}(\Omega),H_0^1(\Omega))} + \int_{Q_T}  \nabla v .  \nabla w_2= \int_{Q_T} f_2(u,v) w_2,
\end{align}
and 
\begin{equation}
\label{initNL}
(u,v)(0,.) = (u_0,v_0)\ \mathrm{in}\ L^{\infty}(\Omega)^2.
\end{equation}
\indent Let $(u,v) \in (W_T \cap L^{\infty}(Q_T))^2$ and $(\widetilde{u},\widetilde{v}) \in (W_T\cap L^{\infty}(Q_T))^2$ be two solutions of \eqref{systemef1f2}. Then, $(u,v) =(\widetilde{u},\widetilde{v})$.
\end{defprop}
\begin{proof}
The nonlinearities $f_1$ and $f_2$ are in $C^{\infty}(\R^2, \R)$, thus they are locally Lipschitz on $\R^2$. This provides the uniqueness of the solution of \eqref{systemef1f2} in $L^{\infty}(Q_T)$ (associated to an initial data $(u_0,v_0) \in L^{\infty}(\Omega)^2$ and a control $h \in L^2(Q_T)$) by a Gronwall argument (see \Cref{proofuniqueness}).
\end{proof}

\subsection{Main results}
Our first main result is the following one.
\begin{theo}\label{mainresult1}
The system \eqref{systemeodd} is globally null-controllable (in the sense of \Cref{deflocalglobalcontr}).\\
\indent More precisely, there exists $(C_p)_{p \in [2,+\infty)} \in (0,\infty)^{[2,+\infty)}$ such that for every initial data $(u_0,v_0) \in L^{\infty}(\Omega)^2$, there exists a control $h \in \bigcap\limits_{p \in [2, +\infty)}L^p(Q_T)$ satisfying
\begin{equation}\label{estimcontrolmr1}
\forall p \in [2, +\infty),\ \norme{h}_{L^p(Q_T)} \leq C_p \left(\norme{u_0}_{L^{\infty}(\Omega)} + \norme{v_0}_{L^{\infty}(\Omega)}^{1/(2k+1)}\right),
\end{equation}
and the solution $(u,v)$ of \eqref{systemeodd} verifies
$$(u,v)(T,.)=(0,0).$$
\end{theo}
\begin{rmk}
We give some natural extensions of this result in \Cref{geneodd}.
\end{rmk}
\begin{rmk}\label{rmkhLinfty}
If we assume that 
\begin{equation}
\Omega \in C^{2, \alpha},
\label{OmegaC2a}
\end{equation}
with $0 < \alpha < 1$, then \Cref{mainresult1} remains true with a control $h \in L^{\infty}(Q_T)$ and the estimate \eqref{estimcontrolmr1} holds true with $p = +\infty$ (see \Cref{rmk2hlinfty} and \Cref{rmk3hlinfty}).
\end{rmk}
Our second main result is a local controllability result for more general reaction-diffusion systems than \eqref{systemeodd}. 
\begin{theo}\label{mainresult2}
Let $(g_1,g_2)\in C^{\infty}(\R;\R)^2$ be such that
\begin{align*}
&g_1(0) = g'_1 (0) = \dots = g_1^{(2k)}(0)=0\ \text{and}\ g_1^{(2k+1)}(0) \neq 0,\\
&g_2(0) \neq 0. \end{align*}
Let $f_1 \in C^{\infty}(\R^2;\R)$, $f_2 \in C^{\infty}(\R^2;\R)$ be such that
\begin{align}
&\forall v \in \R,\  f_1(0,v)=0\label{hyp1},\\
&\forall (u,v) \in \R^2,\  f_2 (u,v) := g_1(u)g_2(v).
\label{hyp2}
\end{align}
Then, the system \eqref{systemef1f2} is locally null-controllable (in the sense of \Cref{deflocalglobalcontr}).
\end{theo}
\begin{app}
By taking $f_1(u,v) = -u^{2k+1}=-f_2(u,v)$,  \Cref{mainresult2} shows the local null-controllability of a model for the non reversible chemical reaction (according to the law of mass action and the Fick's law) 
\begin{equation*}
(2k+1)U\rightharpoonup V,
\end{equation*}
where $u$ and $v$ denote respectively the concentrations of the component $U$ and $V$.\\
\indent However, we can't deduce from \Cref{mainresult2} a local null-controllability result (which is true for $k=1$ thanks to \cite[Theorem 1]{CGR}) of a model for the reversible chemical reaction \begin{equation*}
(2k+1)U\rightleftharpoons V,
\end{equation*}
which corresponds to $f_1(u,v) = -u^{2k+1} + v = -f_2(u,v)$.\\
\indent By taking $f_1(u,v) = (k_2 - (2k_1+1)) u^{2k_1+1} + (k_5 -  (2k_1+1))u^{2k_1+1} v^{k_4}$ and $f_2(u,v) = k_3 u^{2k_1+1} +(k_6 - k_4) u^{2k_1+1} v^{k_4}$ with $k_1$, $k_2$, $k_3$, $k_4$, $k_5$, $k_6$ positive integers, \Cref{mainresult2} shows the local null-controllability of a model for the chemical reaction
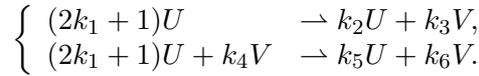
\begin{align*}
\left\{
\begin{array}{l l}
(2k_1+1) U &\rightharpoonup k_2 U + k_3 V,\\
(2k_1+1) U + k_4 V &\rightharpoonup k_5 U + k_6 V.
\end{array}
\right.
\end{align*}
\end{app}

From now on, unless otherwise specified, we denote by $C$ (respectively $C_r$) a positive constant (respectively a positive constant which depends on the parameter r) that may change from line to line.

\section{Global null-controllability for the “odd power system”}\label{preuveglobale}

The aim of this part is to prove \Cref{mainresult1}. We now fix $(u_0,v_0) \in L^{\infty}(\Omega)^2$ until the end of the section.

\subsection{First step of the proof: steer $u$ to $0$}\label{fsetape}

First, we find a control of \eqref{systemeodd} which steers $u$ to $0$ in time $T/2$. 
\begin{prop}\label{controlT/2}
There exists $h_1 \in L^{\infty}((0,T/2)\times \Omega)$ satisfying 
\begin{equation}
\norme{h_1}_{L^{\infty}((0,T/2)\times \Omega)} \leq C \norme{u_0}_{L^{\infty}(\Omega)},
\label{esticontrolT/2}
\end{equation}
such that the solution $(u_1,v_1) \in L^{\infty}((0,T/2)\times \Omega)^2 $ of
\begin{equation*}
\left\{
\begin{array}{l l}
\partial_t u_1 -  \Delta u_1 =   h_1 1_{\omega} &\mathrm{in}\ (0,T/2)\times\Omega,\\
\partial_t v_1 -  \Delta v_1 = u_1^{2k+1}  &\mathrm{in}\ (0,T/2)\times\Omega,\\
u_1,v_1= 0&\mathrm{on}\ (0,T/2)\times\partial\Omega,\\
(u_1,v_1)(0,.)=(u_0,v_0)& \mathrm{in}\ \Omega,
\end{array}
\right.
\label{systemeoddT/2}
\end{equation*}
satisfies $u_1(T/2,.)=0$. Moreover, we have 
\begin{equation}\label{estiv1T/2}
\norme{v_1(T/2,.)}_{L^{\infty}(\Omega)} \leq C \left(\norme{u_0}_{L^{\infty}(\Omega)}^{2k+1} + \norme{v_0}_{L^{\infty}(\Omega)}\right).
\end{equation}
\end{prop}
\begin{proof}
The proof is based on the following result (see \Cref{refpreuve} for some references).
\begin{prop}\label{nullheat}[Null-controllability in $L^{\infty}$ of the linear heat equation in any time]\\
For every $\tau >0$, $y_0 \in L^{\infty}(\Omega)$, there exists $h_\tau \in L^{\infty}(Q_{\tau})$ satisfying
\begin{equation}
\norme{h}_{L^{\infty}((0,\tau)\times \Omega)} \leq C_{\tau} \norme{y_0}_{L^{\infty}(\Omega)},
\end{equation}
such that the solution $y \in L^{\infty}(Q_{\tau})$ of 
\begin{equation}
\left\{
\begin{array}{l l}
\partial_t y -  \Delta y =   h_\tau 1_{\omega} &\mathrm{in}\ (0,\tau)\times\Omega,\\
y= 0&\mathrm{on}\ (0,\tau)\times\partial\Omega,\\
y(0,.)=y_0& \mathrm{in}\ \Omega,
\end{array}
\right.
\label{heat}
\end{equation}
satisfies $y(\tau,.)=0.$
\end{prop}
We use \Cref{nullheat} by taking $\tau = T/2$, $y_0 = u_0$. We get the existence of a control $h_1 \in L^{\infty}((0,T/2)\times \Omega)$ satisfying \eqref{esticontrolT/2} which steers $u_1 \in L^{\infty}((0,T/2)\times \Omega)$ to $0$:
\begin{equation}
\left\{
\begin{array}{l l}
\partial_t u_1 -  \Delta u_1 =   h_1 1_{\omega} &\mathrm{in}\ (0,\tau)\times\Omega,\\
u_1= 0&\mathrm{on}\ (0,\tau)\times\partial\Omega,\\
(u_1(0,.),u_1(T/2,.))=(u_0,0)& \mathrm{in}\ \Omega.
\end{array}
\right.
\label{heatu1}
\end{equation}
Moreover, from \eqref{heatu1}, \eqref{estlinftyfaible} with $p=+\infty$ and \eqref{esticontrolT/2}, we get
\begin{equation}
\norme{u_1}_{L^{\infty}((0,T/2)\times \Omega)} \leq C \norme{u_0}_{L^{\infty}(\Omega)}.
\label{estu1linfty}
\end{equation}
Then, we set $v_1 \in L^{\infty}((0,T/2)\times \Omega)$ (see \Cref{wpl2} with $p= +\infty$), as the solution of 
\begin{equation}
\left\{
\begin{array}{l l}
\partial_t v_1 -  \Delta v_1 = u_1^{2k+1}  &\mathrm{in}\ (0,T/2)\times\Omega,\\
v_1= 0&\mathrm{on}\ (0,T/2)\times\partial\Omega,\\
v_1(0,.)=v_0& \mathrm{in}\ \Omega.
\end{array}
\right.
\label{systemevT/2}
\end{equation}
From \eqref{systemevT/2}, \eqref{estlinftyfaible} with $p=+\infty$ and \eqref{estu1linfty}, we have \eqref{estiv1T/2}.
\end{proof}
\begin{rmk}\label{refpreuve}
There exists at least three proofs of \Cref{nullheat}. First, the common argument is the \textit{null-controllability of the heat equation in $L^2$} proved independently by Gilles Lebeau, Luc Robbiano in 1995 (see \cite{LR} and \cite{LLR}) and Andrei Fursikov, Oleg Imanuvilov in 1996 (see \cite{FI}). Then, the goal is to get a control in $L^{\infty}$. The first method has been employed for the first time by Enrique Fernandez-Cara and Enrique Zuazua (see \cite[Theorem 3.1]{FCZ}) and it is based on the \textit{local regularizing effect of the heat equation} which leads to a \textit{refined observability inequality} (see \cite[Proposition 3.2]{FCZ}). The second method has been employed for the first time by Viorel Barbu (see \cite{B2}) and it is based on a \textit{Penalized Hilbert Uniqueness Method} (see also \cite[Section 3.1.2]{CGR}). The more recent method is due to Olivier Bodart, Manuel Gonzalez-Burgos, Rosario Pérez-Garcia (see \cite{BGBPG}) and it is sometimes called the \textit{fictitious control method} (see \cite[Section 2]{FCGBGP} for the Neumann case and \cite{DL1}).
\end{rmk}

\subsection{Second step of the proof: steer $v$ to $0$ thanks to a control which is as an odd power}\label{scetape}
The aim of this part is to find a control of \eqref{systemeodd} which steers $v$ to $0$ (and $u$ from $0$ to $0$) in time $T$.
\begin{prop}\label{controlT}
Let $((u_1,v_1),h_1)$ as in \Cref{controlT/2}.\\
\indent There exists a control $h_2 \in \bigcap_{p \in [2, +\infty)}L^p((T/2,T)\times \Omega)$ satisfying
\begin{equation}
\forall p \in [2,+\infty),\ \norme{h_2}_{L^{p}((T/2,T)\times \Omega)} \leq C_p\left( \norme{u_0}_{L^{\infty}(\Omega)} + \norme{v_0}_{L^{\infty}(\Omega)}^{1/(2k+1)}\right),
\label{esticontrolT}
\end{equation}
such that the solution $(u_2,v_2) \in L^{\infty}((T/2,T)\times \Omega)^2 $ of
\begin{equation*}
\left\{
\begin{array}{l l}
\partial_t u_2 -  \Delta u_2 =   h_2 1_{\omega} &\mathrm{in}\ (T/2,T)\times\Omega,\\
\partial_t v_2 -  \Delta v_2 = u_2^{2k+1}  &\mathrm{in}\ (T/2,T)\times\Omega,\\
u_2,v_2= 0&\mathrm{on}\ (T/2,T)\times\partial\Omega,\\
(u_2,v_2)(T/2,.)=(0,v_1(T/2,.))& \mathrm{in}\ \Omega,
\end{array}
\right.
\label{systemeoddT}
\end{equation*}
satisfies $(u_2,v_2)(T,.)=(0,0)$.
\end{prop}
\indent Our approach consists in looking at the second equation of \eqref{systemeodd} like a controlled heat equation where the \textit{state} is $v(t,.)$ and the \textit{control input} is $u^{2k+1}(t,.)$. Here, \textbf{the question consists in proving that the heat equation is null-controllable with localized control which \textit{is as an odd power}} of a regular function.

\indent For the sequel, we need to introduce some usual definitions and properties.
\begin{defi}\label{defracinepuissimp}
The mapping $x \in \R \mapsto x^{2k+1} \in \R$ is one-to-one. We note its inverse function $x  \mapsto x^{\frac{1}{2k+1}} $.
\end{defi}
\begin{defi}
\label{defiXp}
For all $\tau > 0$, $0< \tau_1 < \tau_2$, $p \in [1,+\infty]$, we introduce the functional spaces
$$X_{\tau,p} = L^p(0,\tau;W^{2,p}\cap W_0^{1,p}(\Omega)) \cap W^{1,p}(0,\tau; L^p(\Omega)),$$
$$X_{(\tau_1, \tau_2),p} = L^p(\tau_1,\tau_2;W^{2,p}\cap W_0^{1,p}(\Omega)) \cap W^{1,p}(\tau_1,\tau_2; L^p(\Omega)),$$
\end{defi}
The following result is \textit{new} and it is the \textit{key point} of this section.
\begin{prop}\label{lemcontreg}
For every $\tau > 0$, there exists $C_{\tau} >0$ such that \\for every $y_0 \in L^{\infty}(\Omega)$, there exists a control $h_{\tau} \in L^{\infty}(Q_{\tau})$ which verifies 
\begin{align*}
& \norme{h_{\tau}^{\frac{1}{2k+1}}}_{L^{\infty}(Q_{\tau})} \leq C_{\tau} \norme{y_0}_{L^{\infty}(\Omega)}^{1/(2k+1)},\\
&h_{\tau}^{\frac{1}{2k+1}} \in \bigcap\limits_{p \in [2, +\infty)} X_{\tau,p},\\
& \forall p \in [2,+\infty),\ \exists C_{\tau,p} >0,\  \norme{h_{\tau}^{\frac{1}{2k+1}}}_{X_{\tau,p}} \leq C_{\tau,p} \norme{y_0}_{L^{\infty}(\Omega)}^{1/(2k+1)},\\ 
&h_{\tau}(0,.)=h_{\tau}(\tau,.)=0,\\
&\forall t \in [0,\tau],\ supp(h_{\tau}(t,.)) \subset\subset \omega,
\end{align*} 
such that the solution $y \in  L^{\infty}(Q_{\tau})$ of \eqref{heat} satisfies $y(\tau,.)=0$.
\end{prop}
\begin{rmk}
\Cref{lemcontreg} extends \Cref{nullheat}. Its proof is inspired by the \textit{Penalized Hilbert Uniqueness Method} introduced by Barbu (see \cite{B2}). It is based on the Reflexive Uniqueness Method.
\end{rmk}
\begin{rmk}\label{rmk2hlinfty}
If we assume that \eqref{OmegaC2a} holds true, then we can replace $\bigcap\limits_{p \in [2, +\infty)} X_{\tau,p}$ by $C^{1,2}(\overline{Q_T})$. It easily gives the proof of \Cref{rmkhLinfty} by adapting the proof of \Cref{controlT}.
\end{rmk}
Before proving this proposition (whose proof is reported to \Cref{existencecontrolodd}), we apply it to our problem.

\begin{proof}
We apply \Cref{lemcontreg} with $(0,\tau)\leftarrow(T/2,T)$, $y_0 \leftarrow v_1(T/2,.) \in L^{\infty}(\Omega)$. Then, there exists a control $H\in L^{\infty}((T/2,T)\times\Omega)$ such that 
\begin{align}
&H^{\frac{1}{2k+1}} \in \bigcap_{p \in [2, +\infty)} X_{(T/2,T),p},\label{Hreg}\\
&\forall p \in [2,+\infty),\ \norme{H^{\frac{1}{2k+1}}}_{X_{(T/2,T),p}} \leq C_{p} \norme{v_1(T/2,.)}_{L^{\infty}(\Omega)}^{1/(2k+1)},\label{Hregesti}\\
& H(T/2,.) = H(T,.) = 0\label{Hnul},\\
& \forall t \in [T/2,T],\ supp(H(t,.)) \subset\subset \omega\label{suppH},
\end{align}
and the solution $v_2$ of 
\begin{equation}
\left\{
\begin{array}{l l}
\partial_t v_2 -  \Delta v_2 = H  &\mathrm{in}\ (T/2,T)\times\Omega,\\
v_2= 0&\mathrm{on}\ (T/2,T)\times\partial\Omega,\\
v_2(T/2,.)=v_1(T/2,.)& \mathrm{in}\ \Omega,
\end{array}
\right.
\label{systemevT}
\end{equation}
satisfies \begin{equation}v_2(T,.) = 0.\label{nulvT}\end{equation}
From \eqref{Hreg} and a Sobolev embedding (see for instance \cite[Theorem 1.4.1]{WYW} or \cite[Lemma 3.3, page 80]{LSU}), we set 
\begin{equation}
u_2 := H^{\frac{1}{2k+1}} \in \left(\bigcap_{p \in [2, +\infty)} X_{(T/2,T),p}\right) \subset L^{\infty}((T/2,T)\times\Omega).
\label{defu2}
\end{equation}
From \eqref{Hnul} and \eqref{defu2}, we have
\begin{equation}
u_2(T/2,.)= u_2(T,.) = 0.
\label{u2nulext}
\end{equation}
Then, we set, from \eqref{defu2}
\begin{equation}
 h_2:= \partial_t u_2 - \Delta u_2 \in \bigcap_{p \in [2, +\infty)}L^p((T/2,T)\times\Omega),
\label{defh2}
\end{equation}
which is supported in $(T/2,T)\times\omega)$ by \eqref{suppH}.
Moreover, from \eqref{Hregesti} and \eqref{estiv1T/2}, we get
\begin{equation}\label{esticontrolTbis}
\forall p \in [2,+\infty),\ \norme{h_2}_{L^p((T/2,T)\times\Omega)} \leq C_p\left( \norme{u_0}_{L^{\infty}(\Omega)} + \norme{v_0}_{L^{\infty}(\Omega)}^{1/(2k+1)}\right).
\end{equation}
By using \eqref{suppH}, \eqref{systemevT}, \eqref{nulvT}, \eqref{defu2}, \eqref{u2nulext}, \eqref{defh2} and \eqref{esticontrolTbis}, we check that $((u_2,v_2),h_2)$ satisfies \Cref{controlT}.
\end{proof}

\subsection{Strategy of control in the whole interval $(0,T)$}
We gather \Cref{controlT/2} and \Cref{controlT} to find a control which steers $(u,v)$ to $(0,0)$ in time $T$.
\begin{prop}\label{strategyofcontrol0-T}
There exists a control $h \in \bigcap_{p \in [2, +\infty)}L^p(Q_T)$ satisfying 
\begin{equation}
\forall p \in [2,+\infty),\ \norme{h}_{L^{p}((0,T)\times \Omega)} \leq C_p\left( \norme{u_0}_{L^{\infty}(\Omega)} + \norme{v_0}_{L^{\infty}(\Omega)}^{1/(2k+1)}\right),
\label{esticontrolwholeint}
\end{equation} 
such that the solution $(u,v) \in L^{\infty}(Q_T)$ of \eqref{systemeodd} satisfies $(u,v)(T,.) = (0,0)$.
\end{prop}
\begin{proof}
Let $((u_1,v_1),h_1) \in L^{\infty}((0,T/2)\times\Omega)^3$ as in \Cref{controlT/2}.\\
\indent  Let ${((u_2,v_2),h_2) \in L^{\infty}((T/2,T)\times\Omega)^2 \times \bigcap_{p \in [2, +\infty)} L^p((T/2,T)\times\Omega)}$ as in \Cref{controlT}.\\
\indent We define $((u,v),h) \in (W_T \cap L^{\infty}(Q_T))^2 \times \bigcap_{p \in [2, +\infty)}L^p(Q_T)$ by
\begin{align*}
\begin{array}{llll}
u= u_1& \text{in}\ [0,T/2]\times\Omega,& u=u_2&\text{in}\ [T/2,T]\times\Omega,\\
v= v_1& \text{in}\ [0,T/2]\times\Omega,& v=v_2& \text{in}\ [T/2,T]\times\Omega,\\
h = h_1& \text{in}\ (0,T/2)\times\Omega,& h=h_2& \text{in}\ (T/2,T)\times\Omega.
\end{array}
\end{align*}
\indent We deduce from \Cref{controlT/2} and \Cref{controlT} that $(u,v) \in (W_T \cap L^{\infty}(Q_T))^2$ is the solution of \eqref{systemeodd} associated to the control $h$ and $(u,v)(T,.) = (0,0)$. Moreover, from the bounds \eqref{esticontrolT/2} and \eqref{esticontrolT}, we get the bound \eqref{esticontrolwholeint}.
\end{proof}
\Cref{strategyofcontrol0-T} proves our first main result: \Cref{mainresult1}.
\section{A control for the heat equation which is an odd power}\label{existencecontrolodd}

The goal of this section is to prove \Cref{lemcontreg}. We assume in the following that $\tau = T$. First, we prove a \textit{new} Carleman estimate in $L^{2k+2}$ for the heat equation. This type of inequality comes from the \textit{usual Carleman inequality in $L^2$} and \textit{parabolic regularity}. Then, we get the existence of a control for the heat equation such that $h^{\frac{1}{2k+1}}$ is regular by considering a \textit{penalized problem in $L^{\frac{2k+2}{2k+1}}$}, which is a generalization of the \textit{usual Penalized Hilbert Uniqueness Method}.

\subsection{A Carleman inequality in $L^{2k+2}$}\label{carlestimate2k+2subsection}

\subsubsection{Maximal regularity and Sobolev embeddings}
We have the following parabolic regularity result and Sobolev embedding lemma.
\begin{prop}\label{wplp}\cite[Theorem 2.1]{DHP}\\
Let $1 < p < +\infty$, $g \in L^p(Q_T)$, $y_0 \in C_0^{\infty}(\Omega)$. The following Cauchy problem admits a unique solution $y \in X_{T,p}$ (see \Cref{defiXp})
\[
\left\{
\begin{array}{l l}
\partial_t y -  \Delta y= g&\mathrm{in}\ (0,T)\times\Omega,\\
y = 0 &\mathrm{on}\ (0,T)\times\partial\Omega,\\
y(0,.)=y_0 &\mathrm{in}\  \Omega.
\end{array}
\right.
\]
Moreover, if $y_0=0$, there exists $ C >0$ independent of $g$ such that
\[ \norme{y}_{X_{T,p}} \leq C \norme{g}_{L^p(Q_T)}.\]
\end{prop}
\begin{lem}\label{injsobo}\cite[Lemma 3.3, page 80]{LSU}\\
Let $p\in [1,+\infty)$, we have
\[ 
X_{T,p} \hookrightarrow 
\left\{
\begin{array}{c l}
L^{\frac{(N+2)p}{N+2-p}}(0,T;W_0^{1,\frac{(N+2)p}{N+2-p}}(\Omega)) & \mathrm{if}\  p < N+2,\\
L^{2p}(0,T;W_0^{1,2p}(\Omega)) & \mathrm{if}\  p = N+2,\\
L^{\infty}(0,T;W_0^{1,\infty}(\Omega)) & \mathrm{if}\  p > N+2.
\end{array}
\right.
\]
\end{lem}

\subsubsection{Carleman estimates}
We define
\begin{equation}
\forall t \in (0,T),\ \eta(t) := \frac{1}{t(T-t)}.
\end{equation}
Let $\omega_1$ be a nonempty open subset such that
\begin{equation}
\omega_1 \subset \subset \omega.
\end{equation}
Let us recall the usual Carleman estimate in $L^2$ (see \cite[Lemma 8]{CGR} or \cite{FCG} for a general introduction to Carleman estimates).
\begin{prop}\label{carl2lem}\cite[Lemma 8]{CGR}\\
There exist $C >0$ and a function $\rho \in C^2(\overline{\Omega};(0,+\infty))$ such that\\ for every $\varphi_T \in C_0^{\infty}(\Omega)$ and for every $s \geq C$, the solution $\varphi \in X_{T,2}$ of
\begin{equation}
\left\{
\begin{array}{l l}
-\partial_t \varphi - \Delta\varphi = 0 &(0,T)\times\Omega,\\
\varphi = 0 &(0,T)\times\partial\Omega,\\
\varphi(T,.)=\varphi_T &\Omega,
\end{array}
\right.
\label{equationphiadj}
\end{equation}
satisfies
\begin{align}
\label{carl2}
&\int\int_{(0,T)\times\Omega} e^{-s \rho(x) \eta(t)}((s\eta)^3 |\varphi|^2 + (s \eta) |\nabla \varphi |^2) dxdt\\
& \leq C\int\int_{(0,T)\times\omega_1} e^{-s \rho(x) \eta(t)}(s\eta)^3|\varphi|^2 dxdt\notag.
\end{align}
\end{prop}
From now on,  $\rho$ is as in \Cref{carl2lem}. We will deduce from the above $L^2$-Carleman estimate the following $L^{2k+2}$-Carleman estimate.
\begin{theo}\label{carl2+2cor}There exist $C >0$ and $m \in (0,+\infty)$ such that for every $\varphi_T \in C_0^{\infty}(\Omega)$ and for every $s \geq C$, the solution $\varphi \in X_{T,2k+2}$ of \eqref{equationphiadj} satisfies
\begin{align}\label{carl2k+2}
& \int\int_{(0,T)\times\Omega} e^{-(k+1)s \rho(x) \eta(t)}((s\eta)^{-(k+1)m} |\varphi|^{2k+2} + (s \eta)^{-(k+1)(m+2)} |\nabla \varphi |^{2k+2})dxdt\\
& \leq C\int\int_{(0,T)\times\omega_1} e^{-(k+1)s \rho(x) \eta(t)}(s\eta)^{3(k+1)}|\varphi|^{2k+2} dxdt.\notag
\end{align}
\end{theo}
\begin{proof}
Let $\varphi_T \in C_0^{\infty}(\Omega)$, $\varphi \in X_{T,2k+2}$ be the solution of \eqref{equationphiadj} and $s \geq C$ where $C$ is as in \Cref{carl2lem}. \\
\indent By a standard parabolic regularity argument, one may deduce from the $L^2$-Carleman estimate \eqref{carl2} another Carleman type inequality in $L^{p_0}$ with $p_0 = 2(N+2)/N$. If $p_0 > 2k+2$, this estimate implies \eqref{carl2k+2}. Otherwise, we iterate this strategy.\\

\indent  \textbf{Step 1:}
We introduce the sequence $(\psi_n)_{n \geq -1}$,
\begin{equation}
\psi_{-1} := e^{-s \rho \eta /2} (s \eta)^{3/2} \varphi,\qquad \forall n \geq 0,\ \psi_{n} := (s \eta)^{-2} \psi_{n-1} := e^{-s \rho \eta /2}(s \eta)^{3/2 - 2 (n+1)} \varphi.
\label{defpsisuite}
\end{equation}
Then, we also define an increasing sequence $(p_n)_{n \geq -1}$ by
\begin{equation}
p_{-1} := 2,\qquad \forall n \geq 0,\   p_{n} :=
\left\{
\begin{array}{c l}
\frac{(N+2)p_{n-1}}{N+2-p_{n-1}} & \mathrm{if}\  p_{n-1} < N+2,\\
2 p_{n-1} & \mathrm{if}\  p_{n-1} = N+2,\\
+\infty & \mathrm{if}\  p_{n-1} > N+2.\\
\end{array}
\right.
\label{defsequencep}
\end{equation}
Clearly, there exists a unique integer $n_0$ such that
\begin{equation}
p_{n_{0}} > 2k+2 \geq p_{n_0 -1}.
\label{defn0}
\end{equation} 
We will need this technical lemma.
\begin{lem}\label{lemtechniquerec}
For every integer $n \in \N$,
\begin{equation}
\left\{
\begin{array}{l l}
- \partial_t \psi_{n} -  \Delta \psi_{n} = g_n &\mathrm{in}\ (0,T)\times\Omega,\\
\psi_{n}= 0&\mathrm{on}\ (0,T)\times\partial\Omega,\\
\psi_{n}(T,.)=0& \mathrm{in}\ \Omega,
\end{array}
\right.
\label{systemepsin}
\end{equation} 
with 
\begin{equation}
g_n(t,x) = a_n(t,x) \psi_{n-1} + (s \eta)^{-1} \nabla \psi_{n-1} . \nabla \rho, \qquad \text{where}\ \norme{a_n}_{L^{\infty}(Q_T)} \leq C_n.
\label{deffn}
\end{equation}
\end{lem}
\begin{proof} We prove \Cref{lemtechniquerec} by induction on $n \in \N$.\\
\indent We introduce the notation
\begin{equation}
\forall (t,x) \in (0,T)\times\Omega,\ \Theta(t,x) :=e^{-s \rho(x) \eta(t)/2} (s \eta(t))^{3/2}.
\label{defTheta}
\end{equation}
\indent \textit{Initialization}: For $n=0$, by using \eqref{defpsisuite}, \eqref{defTheta} and \eqref{equationphiadj}, we have
\begin{align}
&- \partial_t \psi_0 - \Delta \psi_0\notag\\
&= -\partial_t ((s \eta)^{-2} \psi_{-1}) - \Delta ((s \eta)^{-2} \psi_{-1})\notag\\
&= - \partial_t((s \eta)^{-2}) \psi_{-1} + (s \eta)^{-2} (- \partial_t \psi_{-1} - \Delta \psi_{-1} )\notag\\
&=  - \partial_t((s \eta)^{-2}) \psi_{-1} + (s \eta)^{-2} (- (\partial_t \Theta) \varphi  + \Theta (-\partial_t \varphi - \Delta \varphi) - 2 \nabla \Theta . \nabla \varphi - (\Delta \Theta) \varphi)\notag,
\end{align}
\begin{equation}
\label{hyprecpsi0} - \partial_t \psi_0 - \Delta \psi_0 = - \partial_t((s \eta)^{-2}) \psi_{-1} + (s \eta)^{-2} (- (\partial_t \Theta) \varphi  - 2 \nabla \Theta . \nabla \varphi - (\Delta \Theta) \varphi).
\end{equation}
Straightforward computations lead to 
\begin{equation}
\partial_t \Theta = e^{-s \rho \eta /2} \left(-\frac{1}{2} (s \eta')(s\eta)^{3/2} \rho + \frac{3}{2}(s \eta')(s \eta)^{1/2}\right),
\label{derivetTheta}
\end{equation}
\begin{equation}
\nabla \Theta = -\frac{1}{2} e^{-s \rho \eta /2} (s \eta)^{5/2} \nabla \rho,\quad\Delta \Theta = e^{-s \rho \eta /2}\left(\frac{(s \eta)^{7/2}}{4} |\nabla \rho|^2 - \frac{(s\eta)^{5/2}}{2} \Delta \rho\right).
\label{derivexTheta}
\end{equation}
By using \eqref{hyprecpsi0}, \eqref{derivetTheta}, \eqref{derivexTheta}, we get 
\begin{align}\label{hyprecpsi0suite}
&- \partial_t \psi_0 - \Delta \psi_0\\
&= - \partial_t((s \eta)^{-2}) \psi_{-1} \notag\\
&\quad+\left(e^{-s \rho \eta /2} \left(\frac{1}{2} (s \eta')(s\eta)^{-1/2} \rho - \frac{3}{2}(s \eta')(s \eta)^{-3/2} -\frac{(s \eta)^{3/2}}{4} |\nabla \rho|^2 + \frac{(s\eta)^{1/2}}{2} \Delta \rho \right)\right) \varphi\notag\\
&\quad + e^{-s \rho \eta /2} (s \eta)^{1/2} \nabla \rho . \nabla \varphi\notag.
\end{align}
Moreover, by using \eqref{defpsisuite} and \eqref{derivexTheta}, we have 
\begin{equation}
\label{gradpsi-1}
\psi_{-1} = e^{-s \rho \eta /2} (s \eta)^{3/2} \varphi \Leftrightarrow \varphi = e^{s \rho \eta /2} (s \eta)^{-3/2} \psi_{-1},
\end{equation}
\begin{align}
\label{gradpsi-1bis}(s \eta)^{-1} \nabla \psi_{-1} . \nabla \rho &= (s \eta)^{-1} \Big((\nabla \Theta. \nabla \rho) \varphi + (\nabla \varphi . \nabla \rho) \Theta\Big)\notag\\
&  = e^{-s \rho \eta /2} \left( - \frac{(s \eta)^{3/2}}{2} |\nabla \rho|^2 \varphi + (s\eta)^{1/2} \nabla \rho. \nabla \varphi \right).
\end{align}
We gather \eqref{hyprecpsi0suite}, \eqref{gradpsi-1} and \eqref{gradpsi-1bis} to get \eqref{systemepsin} and \eqref{deffn} for $n=0$ (remark that $\eta' \leq C(\eta^2+\eta^3)$) with
\begin{equation}
a_0:= -\partial_t ((s\eta)^{-2}) + \eta' \left( \frac{s^{-1}\eta^{-2}}{2} \rho  - \frac{3}{2} s^{-2} \eta^{-3}\right) + \frac{1}{4} |\nabla \rho |^2 + \frac{1}{2} (s \eta)^{-1}\Delta \rho \in L^{\infty}(Q_T),
\label{defa_0}
\end{equation}
\indent \textit{Heredity}: Let $n \geq 1$. We assume that \eqref{systemepsin} and \eqref{deffn} hold true for $n-1$. Then, by using \eqref{defpsisuite}, we have
\begin{align*}
- \partial_t \psi_n - \Delta \psi_n &=  -\partial_t ((s \eta)^{-2} \psi_{n-1}) - \Delta ((s \eta)^{-2} \psi_{n-1})\\
&= - \partial_t((s \eta)^{-2}) \psi_{n-1} + (s \eta)^{-2} (- \partial_t \psi_{n-1} - \Delta \psi_{n-1})\\
& = - \partial_t((s \eta)^{-2}) \psi_{n-1} + (s\eta)^{-2} (a_{n-1} \psi_{n-2} + (s \eta)^{-1} \nabla \psi_{n-2} . \nabla \rho )\\
&= - \partial_t((s \eta)^{-2}) \psi_{n-1} + a_{n-1} \psi_{n-1} + (s \eta)^{-1} \nabla \psi_{n-1} . \nabla \rho.
\end{align*}
Therefore, \eqref{systemepsin} and \eqref{deffn} hold true for $n$ with
\begin{equation}
a_n(t,x) := - \partial_t ((s \eta)^{-2}) + a_{n-1} (t,x) \in L^{\infty}(Q_T).
\label{defrecanan-1}
\end{equation}
This ends the proof of \Cref{lemtechniquerec}.
\end{proof}
\textbf{Step 2:} We show by induction that
\begin{align}
\forall n \in \{0,\dots, n_0\},\ \psi_n \in X_{T,p_{n-1}},\ \norme{\psi_n}_{X_{T,p_{n-1}}}^{p_{n-1}} &\leq C_n  \norme{e^{-s \rho \eta /2} (s \eta)^{3/2} \varphi}_{L^{p_{n-1}}((0,T)\times\omega_1)}^{p_{n-1}}\label{pnpn-1}.
\end{align}
\indent First, we treat the case $n=0$. By using \eqref{deffn} for $n=0$, \eqref{gradpsi-1} and \eqref{gradpsi-1bis}, we remark that 
\begin{align}
\label{ecritureg0}
g_0 &= a_0 \psi_{-1} + (s \eta)^{-1} \nabla \psi_{-1}.\nabla \rho\notag\\ &= a_0 e^{-s \rho \eta /2} (s \eta)^{3/2} \varphi + e^{-s \rho \eta /2} \left( - \frac{(s \eta)^{3/2}}{2} |\nabla \rho|^2 \varphi + (s\eta)^{1/2} \nabla \rho. \nabla \varphi \right).
\end{align}
Then, from \eqref{ecritureg0}, we get that $g_0 \in L^2(Q_T)$ and
\begin{equation}
\norme{g_0}_{L^2(Q_T)}^2 \leq C \int\int_{(0,T)\times\Omega} e^{-s \rho(x) \eta(t)}((s\eta)^3 |\varphi|^2 + (s \eta) |\nabla \varphi |^2)dxdt .
\label{estimateg0}
\end{equation}
Consequently, by \eqref{systemepsin} (for $n=0$), \eqref{estimateg0} and a parabolic regularity estimate (see \Cref{wplp} with $p=2$), we find that
\begin{equation}
\psi_0 \in X_{T,2}\ \text{and}\ \norme{\psi_0}_{X_{T,2}}^2 \leq C \int\int_{(0,T)\times\Omega} e^{-s \rho(x) \eta(t)}((s\eta)^3 |\varphi|^2 + (s \eta) |\nabla \varphi |^2)dxdt.
\label{estimatepsi0}
\end{equation}
Gathering the Carleman estimate in $L^2$ i.e. \eqref{carl2} and \eqref{estimatepsi0}, we have
\begin{align}
\norme{\psi_0}_{X_{T,2}}^2
 \leq C \int\int_{(0,T)\times\omega_1} e^{-s \rho(x) \eta(t)}(s\eta)^3|\varphi|^2 dxdt.
\label{psi0sobestimate4}
\end{align}
This concludes the proof of \eqref{pnpn-1} for $n=0$.\\

Now, we assume that \eqref{pnpn-1} holds true for an integer $n\in\{0,\dots, n_0-1\}$. By \eqref{defsequencep}, a Sobolev embedding (see \Cref{injsobo}) applied to the left hand side of \eqref{pnpn-1}, the embedding $L^{{p_n}}((0,T)\times\omega_1) \hookrightarrow L^{p_{n-1}}((0,T)\times\omega_1)$, applied to right hand side of \eqref{pnpn-1}, we obtain
$$\psi_{n} \in L^{p_{n}}(0,T;W^{1,p_{n}}(\Omega)),$$
\begin{equation}
\norme{\psi_{n}}_{L^{p_{n}}(0,T;W^{1,p_{n}}(\Omega))} \leq C_n\norme{e^{-s \rho \eta /2} (s \eta)^{3/2} \varphi}_{L^{p_{n}}((0,T)\times\omega_1)}.
\label{psin}
\end{equation}
By using the parabolic equation satisfied by $\psi_{n+1}$ i.e. \eqref{systemepsin}, \eqref{deffn} for $(n+1)$, \eqref{psin} and a parabolic regularity estimate (see \Cref{wplp} with $p=p_n$), we get \begin{align*}
\psi_{n+1} \in X_{T,p_{n}}\quad\text{and}\quad  \norme{\psi_{n+1}}_{X_{T,p_{n}}}^{p_{n}} &\leq C_{n+1}  \norme{e^{-s \rho \eta /2} (s \eta)^{3/2} \varphi}_{L^{p_{n}}((0,T)\times\omega_1)}^{p_{n}}.
\end{align*}
This ends the proof of \eqref{pnpn-1}.\\

\textbf{Step 3:} We apply \eqref{pnpn-1} with $n=n_0$ and we use a Sobolev embedding (see \Cref{injsobo}) and \eqref{defsequencep} to get 
$$\psi_{n_0} \in L^{p_{n_0}}(0,T;W^{1,p_{n_0}}(\Omega)),$$
\begin{equation}
\norme{\psi_{n_0}}_{L^{p_{n_0}}(0,T;W^{1,p_{n_0}}(\Omega))} \leq C\norme{e^{-s \rho \eta /2} (s \eta)^{3/2} \varphi}_{L^{p_{n_0-1}}((0,T)\times\omega_1)}.
\label{psin0}
\end{equation}
Recalling the definition of $n_0$ in \eqref{defn0},  and by using \eqref{psin0}, together with the embedding $$L^{p_{n_0}}(Q_T) \hookrightarrow L^{{2k+2}}(Q_T),$$ applied to the left hand side of \eqref{psin0} and the embedding $$L^{{2k+2}}((0,T)\times\omega_1) \hookrightarrow L^{p_{n_0-1}}((0,T)\times\omega_1),$$ applied to the right hand side of \eqref{psin0}, we get
\begin{equation}
\norme{\psi_{n_0}}_{L^{2k+2}(Q_T)}^{p_{n_0-1}} +  \norme{\nabla \psi_{n_0}}_{L^{2k+2}(Q_T)}^{p_{n_0-1}} \leq C  \norme{e^{-s \rho \eta /2} (s \eta)^{3/2} \varphi}_{L^{2k+2}((0,T)\times\omega_1)}^{p_{n_0-1}}.\label{pn0pn0embedd}
\end{equation}
Then, from the definition \eqref{defpsisuite} of $\psi_{n_0}$, we get
\begin{equation}\label{psin0varphi1}
\psi_{n_0} =  e^{-s \rho \eta /2}(s \eta)^{-1/2 - 2 n_0} \varphi,
\end{equation}
\begin{equation}
\nabla \psi_{n_0} =-\frac{1}{2} e^{-s \rho \eta /2} (s \eta)^{1/2 - 2 n_0} \varphi \nabla \rho + e^{-s \rho \eta /2}(s \eta)^{-1/2 - 2 n_0} \nabla \varphi.
\label{psin0varphi2}
\end{equation}
Consequently, we deduce from \eqref{psin0varphi1} and \eqref{psin0varphi2} that
\begin{align}\label{pn0pn0embedd2}
&\norme{e^{-s \rho \eta /2}(s \eta)^{-1/2 - 2 n_0} \varphi}_{L^{2k+2}(Q_T)}^{p_{n_0-1}} +  \norme{e^{-s \rho \eta /2}(s \eta)^{-3/2 - 2 n_0} \nabla \varphi}_{L^{2k+2}(Q_T)}^{p_{n_0-1}}\\
&\leq C \left(\norme{\psi_{n_0}}_{L^{2k+2}(Q_T)}^{p_{n_0-1}} +  \norme{(s \eta)^{-1} \nabla \psi_{n_0}}_{L^{2k+2}(Q_T)}^{p_{n_0-1}}\right)\notag\\
& \leq C  \left(\norme{\psi_{n_0}}_{L^{2k+2}(Q_T)}^{p_{n_0-1}} +  \norme{\nabla \psi_{n_0}}_{L^{2k+2}(Q_T)}^{p_{n_0-1}}\right).\notag
\end{align}
By using \eqref{pn0pn0embedd} and \eqref{pn0pn0embedd2}, we get \eqref{carl2k+2} with $m=4n_0+1$.
\end{proof}

We consider $\chi \in C^{\infty}(\overline{\Omega}; [0,+\infty))$ such that
\begin{equation}
supp(\chi) \subset \subset \omega, \qquad \chi = 1\ \text{in}\ \omega_1, \qquad \chi^{\frac{1}{2k+1}} \in C^{\infty}(\overline{\Omega}; [0,+\infty)).
\label{defchi}
\end{equation}
We deduce from \Cref{carl2+2cor} the following result.
\begin{cor}\label{carl4usefaiblecor}
There exist $C >0$ and $m \in (0,+\infty)$ such that\\ for every $\varphi_T \in L^{2k+2}(\Omega)$ and for every $s \geq C$, the solution $\varphi \in L^{2k+2}(Q_T)$ of \eqref{equationphiadj}
satisfies
\begin{align}\label{carl2k+2usefortpart}
& \int\int_{(0,T)\times\Omega} e^{-(k+1)s \rho(x) \eta(t)}((s\eta)^{-(k+1)m} |\varphi|^{2k+2} + (s \eta)^{-(k+1)(m+2)} |\nabla \varphi |^{2k+2})dxdt\\
& \leq C\int\int_{(0,T)\times\omega} \chi^{2k+2} e^{-(k+1)s \rho(x) \eta(t)}(s\eta)^{3(k+1)}|\varphi|^{2k+2} dxdt,\notag
\end{align}
and
\begin{equation}
\norme{\varphi(0,.)}_{L^{2k+2}(\Omega)}^{2k+2}
\leq C_s \int\int_{(0,T)\times\omega} \chi^{2k+2} e^{-(k+1)s \rho(x) \eta(t)}(s\eta)^{3(k+1)}|\varphi|^{2k+2} dxdt.
\label{carl2k+2usefaible}
\end{equation}
\end{cor}

\begin{rmk}
We insist on the fact that the constant $C$ of the observability inequality \eqref{carl2k+2usefaible} depends on the parameter $s$. It is not the case of \eqref{carl2k+2usefortpart}.
\end{rmk}

\begin{proof}
\textbf{Step 1:} We assume that $\varphi_T \in C_0^{\infty}(\Omega)$. We denote by $\varphi \in X_{T,2k+2}$, the solution of \eqref{equationphiadj}. Then, by \Cref{carl2+2cor} and \eqref{defchi}, for every $s \geq C$, \eqref{carl2k+2usefortpart} holds and in particular,
\begin{align}\label{carl4usefortpartproof}
&\int\int_{(0,T)\times\Omega} e^{-(k+1)s \rho \eta}(s\eta)^{-(k+1)m} |\varphi|^{2k+2} \\&
\leq C\int\int_{(0,T)\times\omega} \chi^{2k+2} e^{-(k+1)s \rho \eta}(s\eta)^{3(k+1)}|\varphi|^{2k+2}\notag
\end{align}
We fix $s$ sufficiently large such that \eqref{carl4usefortpartproof} holds.\\
\indent By using
\begin{equation}
\min_{[T/4,3T/4]\times\overline{\Omega}} e^{-(k+1)s \rho(x) \eta(t)}(s\eta)^{-(k+1)m} > 0,
\label{minoration}
\end{equation}
together with \eqref{carl4usefortpartproof}, we get 
\begin{align}
\int_{T/4}^{3T/4}\int_{\Omega} |\varphi|^{2k+2} dxdt
\leq C_s \int\int_{(0,T)\times\omega} \chi^{2k+2} e^{-(k+1)s \rho(x) \eta(t)}(s\eta)^{3(k+1)}|\varphi|^{2k+2} dxdt.
\label{carl4usefortpart2}
\end{align}
From the dissipation of the $L^{2k+2}$-norm for \eqref{equationphiadj} (obtained by multiplying the equation \eqref{equationphiadj} by $|\varphi|^{p-2} \varphi$ and integrating over $\Omega$): $\norme{\varphi(0,.)}_{L^{2k+2}(\Omega)} \leq \norme{\varphi(t,.)}_{L^{2k+2}(\Omega)}$ for $t \in (T/4,3T/4)$, by integrating in time, we get
\begin{equation}
\norme{\varphi(0,.)}_{L^{2k+2}(\Omega)}^{2k+2} \leq C \int_{T/4}^{3T/4}\int_{\Omega} |\varphi|^{2k+2} dxdt.
\label{dissip}
\end{equation}
Gathering \eqref{carl4usefortpart2} and \eqref{dissip}, we get \eqref{carl2k+2usefaible}.\\

\textbf{Step 2:} The general case comes from a density argument by using in particular \Cref{wpl2}: \eqref{estlinftyfaible} for $p=2k+2$. The complete proof is reported to \Cref{appendixdensity}. 
\end{proof}

\subsection{A new penalized duality method in $L^{(2k+2)/(2k+1)}$, the Reflexive Uniqueness Method}\label{RUM}
From now on, $\chi$ is a function which belongs to $C^{\infty}(\overline{\Omega};[0,+\infty))$ satisfying \eqref{defchi} and $m$, $s$ are fixed by \Cref{carl4usefaiblecor}.\\
\indent We introduce the notations
\begin{equation}
q := \frac{2k+2}{2k+1},
\label{defq}
\end{equation}
\begin{equation*}
L_{wght}^{q}((0,T)\times\omega) := \left\{ h \in L^q((0,T)\times\omega)\ ;\ e^{s \rho \eta/2} (s \eta)^{-3/2} h \in L^{q}((0,T)\times\omega)\right\}.
\end{equation*}
The goal of this section is to get a null-controllability result for the heat equation thanks to the observability inequalities of \Cref{carl4usefaiblecor}. 
\begin{prop}\label{phumL^qprop}
For every $\zeta_0 \in L^q(\Omega)$, there exists a control $h \in L_{wght}^{q}((0,T)\times\omega)$ such that the solution $\zeta$ of
\begin{equation}
\left\{
\begin{array}{l l}
\partial_t \zeta -  \Delta \zeta =   h \chi &\mathrm{in}\ (0,T)\times\Omega,\\
\zeta= 0&\mathrm{on}\ (0,T)\times\partial\Omega,\\
\zeta(0,.)=\zeta_0& \mathrm{in}\ \Omega,
\end{array}
\right.
\label{heatzeta}
\end{equation}
satisfies $\zeta(T,.)=0$ and
\begin{equation}
\norme{e^{s \rho(x) \eta(t)/2} (s \eta)^{-3/2}h}_{L^{q}((0,T)\times\omega)} \leq C \norme{\zeta_0}_{L^{q}(\Omega)}.
\label{estimatehprum}
\end{equation} 
\end{prop}
\begin{proof}
Let $\zeta_0 \in C_0^{\infty}(\Omega)$. The general case comes from a density argument.\\
\indent We first state two easy facts.
\begin{claim}\label{lemmaprimitive}
The antiderivative of the continuous mapping $x \in \R \mapsto x^{\frac{1}{2k+1}}$ (see \Cref{defracinepuissimp}) is the strictly convex function \[x \in \R \mapsto \frac{1}{q}\abs{x}^q := \left\{
\begin{array}{c l}
\frac{1}{q} \exp(q \log (\abs{x})) & \mathrm{if}\  x \neq 0,\\
0 & \mathrm{if}\  x =0.
\end{array}
\right.\]
\end{claim}
\begin{claim}
The real numbers $2k+2$, $q$ belong to $(1, +\infty)$ and are conjugate:
\[\frac{1}{2k+2} + \frac{1}{q} = 1.\]
\end{claim}
Let $\varepsilon \in (0,1)$. We consider the minimization problem
\begin{equation}
 \inf\limits_{h \in L_{wght}^{q}((0,T)\times\omega)} J(h),
 \label{minprob}
\end{equation}
where $J$ is defined as follows: for every $h \in L_{wght}^{q}((0,T)\times\omega)$,
\begin{equation}
J(h) := \frac{1}{q} \int\int_{(0,T)\times\omega} e^{(q/2) s \rho(x) \eta(t)} (s \eta)^{-3q/2} |h|^{q}dxdt + \frac{1}{q\varepsilon} \norme{\zeta(T,.)}_{L^{q}\Omega)}^{q},
 \label{defJ}
 \end{equation}
where $\zeta \in X_{T,q}$ is the solution of \eqref{heatzeta} (see \Cref{wplp}).\\
\indent The mapping $J$ is a coercive, strictly convex (see \Cref{lemmaprimitive}), $C^{1}$ function on the reflexive space $L_{wght}^{q}((0,T)\times\omega)$. Then, $J$ has a unique minimum $h^{\varepsilon}$. We denote by $\zeta^{\varepsilon} \in X_{T,q}$ the solution of \eqref{heatzeta} associated to the control $h^{\varepsilon}$. The Euler-Lagrange equation gives
\begin{align}\label{eulerlagrange}
\forall h \in L_{wght}^{q}((0,T)\times\omega),\
& \int\int_{(0,T)\times\omega}e^{(q/2) s \rho(x) \eta(t)} (s \eta)^{-3q/2} (h^{\varepsilon})^{1/(2k+1)} hdxdt\\  
&+ \frac{1}{\varepsilon} \int_{\Omega}(\zeta^{\varepsilon}(T,x))^{1/(2k+1)}\zeta(T,x)dx = 0\notag,
\end{align}
where $\zeta \in X_{T,q}$ is the solution of \eqref{heatzeta} (associated to the control $h$) with $\zeta_0 = 0$.\\
\indent We introduce $\varphi^{\varepsilon} \in L^{2k+2}(Q_T)$ the solution of the adjoint problem
\begin{equation}
\left\{
\begin{array}{l l}
-\partial_t \varphi^{\varepsilon} - \Delta\varphi^{\varepsilon} = 0 &\mathrm{in}\ (0,T)\times\Omega,\\
\varphi^{\varepsilon} = 0 &\mathrm{on}\ (0,T)\times\partial\Omega,\\
\varphi^{\varepsilon}(T,.)=-\frac{1}{\varepsilon}(\zeta^{\varepsilon}(T,.))^{1/(2k+1)}=:\varphi_T^{\varepsilon} &\mathrm{in}\ \Omega.
\end{array}
\right.
\label{equationphi}
\end{equation}
By a duality argument between $\zeta$, the solution of \eqref{heatzeta} with $\zeta_0 = 0$ and $\varphi^{\epsilon}$, we have
\begin{equation}
\forall h \in L_{wght}^{q}((0,T)\times\omega),\quad \int_{\Omega}\varphi^{\varepsilon}(T,x)\zeta(T,x) dx
 = \int\int_{(0,T)\times\omega}  \varphi^{\varepsilon} h \chi dxdt.
\label{dual1}
\end{equation}
Indeed, first, one can prove the result for $\varphi_T^{\varepsilon} \in C_0^{\infty}(\Omega)$ because in this case $\varphi^{\varepsilon} \in X_{T,2k+2}$ and $\zeta \in X_{T,q}$. This justifies the calculations for the duality argument. Then, the fact that $\overline{C_0^{\infty}(\Omega)}^{L^{2k+2}(\Omega)} =L^{2k+2}(\Omega)$ leads to \eqref{dual1}.\\

\indent From \eqref{equationphi} (definition of $\varphi_T^{\varepsilon}$) and \eqref{dual1}, we have 
\begin{equation}
\forall h \in L_{wght}^{q}((0,T)\times\omega),\ -\frac{1}{\varepsilon} \int_{\Omega}(\zeta^{\varepsilon}(T,x))^{1/(2k+1)}\zeta(T,x)dx = \int\int_{(0,T)\times\omega}  \varphi^{\varepsilon}h \chi dxdt.
\label{dual1bis}
\end{equation}
Then, by using \eqref{eulerlagrange} and \eqref{dual1bis}, we obtain
\begin{equation}
(h^{\varepsilon})^{1/(2k+1)} = e^{-(q/2) s \rho(x) \eta(t)} (s \eta)^{3q/2} \varphi^{\varepsilon} \chi.
\label{hfonctionphi}
\end{equation}
Moreover, from a duality argument between $\varphi^{\varepsilon}$ and $\zeta^{\varepsilon}$, together with \eqref{hfonctionphi}, we have
\begin{align}\label{dual2}
&\int_{\Omega}\varphi^{\varepsilon}(T,x)\zeta^{\varepsilon}(T,x) dx\\
& = - \frac{1}{\varepsilon} \norme{\zeta^{\varepsilon}(T,.)}_{L^{q}(\Omega)}^{q} \notag\\
&=\int\int_{(0,T)\times\omega}  \varphi^{\varepsilon}h^{\varepsilon}\chi dxdt + \int_{\Omega} \varphi^{\varepsilon}(0,x)\zeta_0(x)dx\notag\\
&= \int\int_{(0,T)\times\omega} e^{-(k+1) s \rho\eta} (s\eta)^{3(k+1)}|\varphi^{\varepsilon}|^{2k+2} \chi^{2k+2} + \int_{\Omega} \varphi^{\varepsilon}(0,.)\zeta_0(.).\notag
\end{align}
By Young's inequality, we have for every $\delta > 0$,
\begin{equation}
\left|\int_{\Omega} \varphi^{\varepsilon}(0,x)\zeta_0(x)dx\right| \leq \delta \norme{\varphi^{\varepsilon}(0,.)}_{L^{2k+2}(\Omega)}^{2k+2} + C_{\delta} \norme{\zeta(0,.)}_{L^{q}(\Omega)}^{q}.
\label{young}
\end{equation}
From \eqref{dual2}, \eqref{young}, the observability inequality \eqref{carl2k+2usefaible} (applied to $\varphi^{\varepsilon}$), and by taking $\delta$ sufficiently small, we get
\begin{equation}
\frac{1}{\varepsilon} \norme{\zeta^{\varepsilon}(T,.)}_{L^{q}(\Omega)}^{q} + \int\int_{(0,T)\times\omega} e^{-(k+1) s \rho\eta} (s\eta)^{3(k+1)}|\varphi^{\varepsilon}|^{2k+2} \chi^{2k+2} \leq C \norme{\zeta_0}_{L^{q}(\Omega)}^{q}.
\label{estcontrsol}
\end{equation}
Now, plugging \eqref{hfonctionphi} into \eqref{estcontrsol}, we obtain
\begin{equation}
\frac{1}{\varepsilon} \norme{\zeta^{\varepsilon}(T,.)}_{L^{q}(\Omega)}^{q} + \norme{e^{s \rho(x) \eta(t)/2} (s \eta)^{-3/2}h^{\varepsilon}}_{L^{q}((0,T)\times\omega)}^{q} \leq C \norme{\zeta_0}_{L^{q}(\Omega)}^{q}.
\label{estcontrsolbis}
\end{equation}
In particular, from \eqref{estcontrsolbis}, we have
\begin{equation}
\zeta^{\varepsilon}(T,.) \underset{\varepsilon \rightarrow 0}{\rightarrow} 0\  \mathrm{in}\ L^{q}(\Omega),
\label{convzeta0}
\end{equation}
and
\begin{equation}
\norme{h^{\varepsilon}}_{L^{q}(Q_T)} \leq C.
\label{unebornecontrole}
\end{equation}
We remark that 
\begin{equation}
\zeta^{\varepsilon} = \zeta_1^{\varepsilon} + \zeta_2^{\varepsilon},
\label{decompzeta}
\end{equation}
with $\zeta_1^{\varepsilon}$, $\zeta_2^{\varepsilon}$ satisfying
\begin{equation}
\left\{
\begin{array}{l l}
\partial_t \zeta_1^{\varepsilon} -  \Delta \zeta_1^{\varepsilon} =   0 &\mathrm{in}\ (0,T)\times\Omega,\\
\zeta_1^{\varepsilon}= 0&\mathrm{on}\ (0,T)\times\partial\Omega,\\
\zeta_1^{\varepsilon}=\zeta_0& \mathrm{in}\ \Omega,
\end{array}
\right.
\qquad
\left\{
\begin{array}{l l}
\partial_t \zeta_2^{\varepsilon} -  \Delta \zeta_2^{\varepsilon} =   h^{\varepsilon} \chi &\mathrm{in}\ (0,T)\times\Omega,\\
\zeta_2^{\varepsilon}= 0&\mathrm{on}\ (0,T)\times\partial\Omega,\\
\zeta_2^{\varepsilon}=0& \mathrm{in}\ \Omega.
\end{array}
\right.
\label{heatzeta-1,2}
\end{equation}
Then, by using \eqref{decompzeta}, \eqref{heatzeta-1,2}, \eqref{unebornecontrole} and \Cref{wplp} with $p=q$, we have 
\begin{equation}
\norme{\zeta^{\varepsilon}}_{X_{T,q}} \leq C.
\label{uneborneenzeta}
\end{equation}
So, from \eqref{uneborneenzeta}, up to a subsequence, we can suppose that  there exists $\zeta \in X_{T,q}$ such that
\begin{equation}
\zeta^{\varepsilon} \underset{\varepsilon \rightarrow 0}{\rightharpoonup} \zeta\ \text{in}\ X_{T,q},
\label{wkVeps}
\end{equation}
and from the embedding $X_{T,q} \hookrightarrow C([0,T];L^{q}(\Omega))$ (see \cite[Section 5.9.2, Theorem 2]{E}),
\begin{equation}
\zeta^{\varepsilon}(0,.) \underset{\varepsilon \rightarrow 0}{\rightharpoonup} \zeta(0,.)\ \text{in}\ L^{q}(\Omega),\ \zeta^{\varepsilon}(T,.) \underset{\varepsilon \rightarrow 0}{\rightharpoonup} \zeta(T,.)\ \text{in}\  L^{q}(\Omega).
\label{convcicf}
\end{equation}
Then, as we have $\zeta^{\varepsilon}(0,.)=\zeta_0$ and \eqref{convzeta0}, we deduce that \begin{equation}
\zeta(0,.) = \zeta_0\  \mathrm{and}\  \zeta(T,.) = 0.
\label{cicf}
\end{equation}
Moreover, from \eqref{estcontrsolbis}, up to a subsequence, we can suppose that there exists $h \in L_{wght}^{q}((0,T)\times\omega)$ such that
\begin{equation}
h^{\varepsilon} \underset{\varepsilon \rightarrow 0}\rightharpoonup h \ \text{in}\ L_{wght}^{q}((0,T)\times\omega),
\label{wk*eps}
\end{equation}
and
\begin{align}
\label{estlimiteeps}
\norme{e^{s \rho(x) \eta(t)/2} (s \eta)^{-3/2}h}_{L^{q}((0,T)\times\omega)}^{q} &\leq \underset{\varepsilon \rightarrow 0}\lim\inf\norme{e^{s \rho(x) \eta(t)/2} (s \eta)^{-3/2}h^{\varepsilon}}_{L^{q}((0,T)\times\omega)}^{q}\\
& \leq C \norme{\zeta_0}_{L^{q}(\Omega)}^{q}\notag.
\end{align}
Then, from \eqref{wkVeps}, \eqref{wk*eps}, and \eqref{convcicf}, we let $\varepsilon \rightarrow 0$ in the following equations
\begin{equation}
\left\{
\begin{array}{l l}
\partial_t \zeta^{\varepsilon} -  \Delta \zeta^{\varepsilon} = h^{\varepsilon} \chi &\mathrm{in}\ (0,T)\times\Omega,\\
\zeta^{\varepsilon}= 0&\mathrm{on}\ (0,T)\times\partial\Omega,\\
\zeta^{\varepsilon}(0,.)=\zeta_0& \mathrm{in}\ \Omega,
\end{array}
\right.
\label{heatzetaeps}
\end{equation}
and by using \eqref{cicf}, we deduce
\begin{equation}
\left\{
\begin{array}{l l}
\partial_t \zeta -  \Delta \zeta =   h \chi &\mathrm{in}\ (0,T)\times\Omega,\\
\zeta= 0&\mathrm{on}\ (0,T)\times\partial\Omega,\\
(\zeta(0,.),\zeta(T,.))=(\zeta_0,0)& \mathrm{in}\ \Omega.
\end{array}
\right.
\label{heatzetalim}
\end{equation}
Therefore, \eqref{heatzetalim} and \eqref{estlimiteeps} conclude the proof of \Cref{phumL^qprop}.
\end{proof}

\subsection{A Bootstrap argument}\label{bootstrap}
The goal of this section is to prove \Cref{lemcontreg}. We keep the same notations as in the proof of \Cref{phumL^qprop}. We want to improve the regularity of the sequence $((h^{\varepsilon})^{\frac{1}{2k+1}})_{\varepsilon > 0}$. The \textit{key point} is the equality \eqref{hfonctionphi}. We deduce that the regularity of $(h^{\varepsilon})^{\frac{1}{2k+1}}$ depends on the regularity of $e^{-(q/2) s \rho(x) \eta(t)} (s \eta)^{3q/2} \varphi^{\varepsilon}$. We use \textit{parabolic regularity estimates} (see \Cref{wplp}) and a \textit{bootstrap argument} (similar to the proof of \Cref{carl2+2cor}). The \textit{starting point} is \eqref{carl2k+2usefortpart}.\\

\textbf{Step 1:}  We introduce the sequence $(\psi_n^{\varepsilon})_{n \geq -1}$,
\begin{equation}
\psi_{-1}^{\varepsilon} := e^{-s \rho \eta /2} (s \eta)^{-m/2} \varphi,\quad \forall n \geq 0,\ \psi_{n} := (s \eta)^{-2} \psi_{n-1}^{\varepsilon} = e^{-s \rho \eta /2}(s \eta)^{-m/2 - 2 (n+1)} \varphi^{\varepsilon},
\label{defpsisuitebis}
\end{equation}
where $m$ is defined in \Cref{carl4usefaiblecor}. Then, we also define an increasing sequence $(p_n)_{n \geq -1}$ by
\begin{equation}
p_{-1} := 2k+2,\qquad \forall n \geq 0,\   p_{n} :=
\left\{
\begin{array}{c l}
\frac{(N+2)p_{n-1}}{N+2-p_{n-1}} & \mathrm{if}\  p_{n-1} < N+2,\\
2 p_{n-1} & \mathrm{if}\  p_{n-1} = N+2,\\
+\infty & \mathrm{if}\  p_{n-1} > N+2.\\
\end{array}
\right.
\label{defsequencepbis}
\end{equation}
We denote by $l$ the integer such that
\begin{equation}
l := \min \{n \in \N\ ;\ p_n = +\infty\}.
\label{pinfty}
\end{equation}
By using \eqref{defpsisuitebis} and \eqref{equationphi}, we show by induction that for every $n \in \N$,
\begin{equation}
\left\{
\begin{array}{l l}
- \partial_t \psi_{n}^{\varepsilon} -  \Delta \psi_{n}^{\varepsilon} = g_n^{\varepsilon} &\mathrm{in}\ (0,T)\times\Omega,\\
\psi_{n}^{\varepsilon}= 0&\mathrm{on}\ (0,T)\times\partial\Omega,\\
\psi_{n}^{\varepsilon}(T,.)=0& \mathrm{in}\ \Omega,
\end{array}
\right.
\label{systemepsinbis}
\end{equation} 
with 
\begin{equation}
g_n(t,x) = a_n(t,x) \psi_{n-1}^{\varepsilon} + (s \eta)^{-1} \nabla \psi_{n-1}^{\varepsilon} . \nabla \rho, \qquad \text{where}\ \norme{a_n}_{L^{\infty}(Q_T)} \leq C_n.
\label{deffnbis}
\end{equation}
Indeed, straightforward computations as in the proof of \Cref{lemtechniquerec} lead to 
\begin{equation}
a_0:= -\partial_t ((s\eta)^{-2}) + \eta' \left( \frac{s^{-1}\eta^{-2}}{2} \rho  + \frac{m}{2} s^{-2} \eta^{-3}\right) + \frac{1}{4} |\nabla \rho |^2 + \frac{1}{2} (s \eta)^{-1}\Delta \rho,
\label{defa_0bis}
\end{equation}
\begin{equation}
a_n:=  - \partial_t ((s \eta)^{-2}) + a_{n-1}.
\label{defrecanan-1bis}
\end{equation}

\textbf{Step 2:} From \eqref{carl2k+2usefortpart}, \eqref{estcontrsol}, \eqref{defpsisuitebis} and \eqref{defq}, we have 
\begin{equation}
\norme{\psi_{-1}}_{L^{2k+2}(Q_T)} + \norme{(s\eta)^{-1} \nabla \psi_{-1}}_{L^{2k+2}(Q_T)} \leq C \norme{\zeta_0}_{L^{q}(\Omega)}^{q/(2k+2)} = C\norme{\zeta_0}_{L^{q}(\Omega)}^{1/(2k+1)}.
\label{estipsi-1bis}
\end{equation}
Then, by using parabolic regularity estimate (see \Cref{wplp}), \eqref{systemepsinbis}, \eqref{deffnbis}, \eqref{estipsi-1bis} and an induction argument (as in the proof of \Cref{carl2+2cor}), we have that
\begin{align}
\forall n \in \{0,\dots,l\},\ \psi_n^{\varepsilon} \in X_{T,p_{n-1}}\quad\text{and}\quad  \norme{\psi_n^{\varepsilon}}_{X_{T,p_{n-1}}} &\leq C_n   \norme{\zeta_0}_{L^{q}(\Omega)}^{1/(2k+1)}\label{pnpn-1bis}.
\end{align}

\textbf{Step 3:} We apply \eqref{pnpn-1bis} with $n = l$ (see \eqref{pinfty}) and we use \Cref{injsobo} with $p=p_{l-1}$ to get
\begin{align}
\psi_l^{\varepsilon} \in L^{\infty}(0,T;W_0^{1,\infty}(\Omega))\qquad\text{and}\qquad  \norme{\psi_l^{\varepsilon}}_{L^{\infty}(0,T;W_0^{1,\infty}(\Omega))} &\leq C   \norme{\zeta_0}_{L^{q}(\Omega)}^{1/(2k+1)}\label{pnpn-1tri}.
\end{align}
From a parabolic regularity estimate (see \Cref{wplp}) applied to the heat equation satisfied by $\psi_{l+1}^{\varepsilon}$ and \eqref{pnpn-1tri}, we obtain
\begin{equation}
\psi_{l+1}^{\varepsilon} \in \cap_{p \in [2,+\infty)} X_{T,p}\quad\text{and}\quad \forall p\in [2,+\infty),\ \norme{\psi_{l+1}^{\varepsilon}}_{X_{T,p}} \leq C_p  \norme{\zeta_0}_{L^{q}(\Omega)}^{1/(2k+1)}\label{pnpn-1quad}.
\end{equation}
From \eqref{hfonctionphi}, \eqref{defpsisuitebis} (see in particular that $q > 1$), we have
\begin{equation}
\forall p\in [2,+\infty),\ \norme{(h^{\varepsilon})^{1/(2k+1)}}_{X_{T,p}}\leq C_p \norme{\psi_{l+1}^{\varepsilon}}_{X_{T,p}}.
\label{hxpnlinftyint}
\end{equation}
From \eqref{pnpn-1quad} and \eqref{hxpnlinftyint}, we have
\begin{equation}
(h^{\varepsilon})^{1/(2k+1)} \in \bigcap_{p \in [2,+\infty)} X_{T,p},\ \forall p\in [2,+\infty),\ \norme{(h^{\varepsilon})^{1/(2k+1)}}_{X_{T,p}}  \leq C_p \norme{\zeta_0}_{L^{q}(\Omega)}^{1/(2k+1)}.
\label{hxpnlinfty}
\end{equation}
Now, by \eqref{estcontrsolbis} and \eqref{hxpnlinfty}, we have 
\begin{equation}
\forall p\in [2,+\infty),\ \frac{1}{\varepsilon^{1/(2k+2)}} \norme{\zeta^{\varepsilon}(T,.)}_{L^{q}(\Omega)}^{1/(2k+1)} + \norme{(h^{\varepsilon})^{1/(2k+1)}}_{X_{T,p}} \leq C_p \norme{\zeta_0}_{L^{q}(\Omega)}^{1/(2k+1)}.
\label{hxpnlinftyavecfini}
\end{equation}

\textbf{Step 4:} From \eqref{hxpnlinftyavecfini} and same arguments given as previously (see \Cref{RUM}), together with a diagonal extraction process, up to a subsequence, we can assume that there exist $H \in \cap_{p \in [2,+\infty)}X_{T,p}$ and $\zeta \in X_{T,p}$ such that
\begin{align}
&(h^{\varepsilon})^{1/(2k+1)} \underset{\varepsilon \rightarrow 0}\rightharpoonup H \ \text{in}\ X_{T,p}\ \forall p \in [2,+\infty),
\label{convxpnh}
\\
&(h^{\varepsilon})^{1/(2k+1)} \underset{\varepsilon \rightarrow 0}\rightarrow H \ \text{in}\ L^{\infty}(Q_T),\qquad \Big(\Rightarrow h^{\varepsilon} \underset{\varepsilon \rightarrow 0}\rightarrow H^{2k+1} \ \text{in}\ L^{\infty}(Q_T)\Big),
\label{convlinftyh}
\\
&H(0,.) = 0,\ \quad   H(T,.) = 0,\qquad \big(\text{see}\ \eqref{hfonctionphi}\big),
\label{cicfbootcont}
\\
&\zeta^{\varepsilon} \underset{\varepsilon \rightarrow 0}{\rightharpoonup} \zeta\ \text{in}\ X_{T,q},
\label{wnVepsboot}
\\
&\zeta(0,.) = \zeta_0,\ \qquad   \zeta(T,.) = 0.
\label{cicfbootsol}
\end{align}
The strong $L^{\infty}$-convergence \eqref{convlinftyh} is a consequence of the weak $X_{T,p}$-convergence \eqref{convxpnh} for $p$ sufficiently large because $X_{T,p}$ is relatively compact in $L^{\infty}(Q_T)$ (see \cite[Section 8, Corollary 4]{S}: Aubin-Lions lemma).\\
\indent By using \eqref{convlinftyh}, \eqref{wnVepsboot} and \eqref{cicfbootsol} and by letting $\varepsilon \rightarrow 0$ in the following equations
\begin{equation}
\left\{
\begin{array}{l l}
\partial_t \zeta^{\varepsilon} -  \Delta \zeta^{\varepsilon} = h^{\varepsilon} \chi &\mathrm{in}\ (0,T)\times\Omega,\\
\zeta^{\varepsilon}= 0&\mathrm{on}\ (0,T)\times\partial\Omega,\\
\zeta^{\varepsilon}(0,.)=\zeta_0& \mathrm{in}\ \Omega,
\end{array}
\right.
\label{heatzetaepsboot}
\end{equation}
we deduce
\begin{equation}
\left\{
\begin{array}{l l}
\partial_t \zeta -  \Delta \zeta =   H^{2k+1} \chi &\mathrm{in}\ (0,T)\times\Omega,\\
\zeta= 0&\mathrm{on}\ (0,T)\times\partial\Omega,\\
(\zeta(0,.),\zeta(T,.))=(\zeta_0,0)& \mathrm{in}\ \Omega.
\end{array}
\right.
\label{heatzetalimboot}
\end{equation}

To sum up, for all $\zeta_0 \in C_0^{\infty}(\Omega)$, we have found a control 
\begin{equation}
h: = H^{2k+1} \chi \in L^{\infty}(Q_T),
\label{defhcontrolefinal}
\end{equation} 
such that 
\begin{align} \label{boundcontrlboot}
&h^{1/(2k+1)} = H \chi^{1/(2k+1)} \in  \bigcap_{[2,+\infty)} X_{T,p}, \qquad (\text{see}\ \eqref{convxpnh}),\\
\label{estihXTp}
&\forall p \in [2,+\infty),\  \norme{h^{1/(2k+1)}}_{X_{T,p}} \leq C_p \norme{\zeta_0}_{L^{q}(\Omega)}^{1/(2k+1)},\qquad (\text{see}\ \eqref{hxpnlinftyavecfini}),\\
\label{hnulen0enT}
&h(0,.) = h(T,.) = 0 , \qquad (\text{see}\ \eqref{cicfbootcont}).
\end{align}
Moreover, from \eqref{heatzetalimboot}, the solution $\zeta \in L^{\infty}(Q_T)$ of 
\begin{equation}
\left\{
\begin{array}{l l}
\partial_t \zeta -  \Delta \zeta =   h &\mathrm{in}\ (0,T)\times\Omega,\\
\zeta= 0&\mathrm{on}\ (0,T)\times\partial\Omega,\\
\zeta(0,.)=\zeta_0 & \mathrm{in}\ \Omega,
\end{array}
\right.
\label{heatzetalimbootbis}
\end{equation}
satisfies 
\begin{equation}
\label{zetanulTproofprop}
\zeta(T,.) = 0.
\end{equation}
By \eqref{defhcontrolefinal}, \eqref{boundcontrlboot}, \eqref{estihXTp}, \eqref{hnulen0enT}, \eqref{heatzetalimbootbis} and \eqref{zetanulTproofprop}, we deduce \Cref{lemcontreg}\\ for $\zeta_0 \in C_0^{\infty}(\Omega)$. The general case comes from $\overline{C_0^{\infty}(\Omega)}^{L^{q}(\Omega)} =L^{q}(\Omega)$ and the bound \eqref{boundcontrlboot}.
\begin{rmk}\label{rmk3hlinfty}
In this paragraph, we gives the main details to get \Cref{rmk2hlinfty} and consequently \Cref{rmkhLinfty}. By \eqref{OmegaC2a}, the function $\rho$, as in \Cref{carl2lem}, can be chosen such that \begin{equation}
\rho \in C^{2, \alpha}(\overline{\Omega}),
\label{lemregrho}
\end{equation} 
see the proof of \cite[Lemma 2.68]{C}.\\
\indent Let us take $\beta$ such that $ 1 + \alpha \leq \beta < 2$. From Sobolev embedding (see \cite[Corollary 1.4.1]{WYW}) and \eqref{pnpn-1quad}, we have for $p$ sufficiently large,
\begin{equation}
\psi_{l+1}^{\varepsilon} \in X_{T,p} \hookrightarrow C^{\beta/2, \beta}(\overline{Q_T})\ \text{and}\ \norme{\psi_{l+1}^{\varepsilon}}_{C^{\beta/2, \beta}(\overline{Q_T})} \leq C \norme{\zeta_0}_{L^{q}(\Omega)}^{1/(2k+1)}.
\label{injholder}
\end{equation}
From \eqref{lemregrho}, \eqref{defa_0bis} and \eqref{defrecanan-1bis}, we have $
a_{l+2} \in C^{\alpha/2, \alpha}(\overline{Q_T})$. Then, we deduce from \eqref{systemepsinbis}, \eqref{deffnbis} for $n=l+2$, \eqref{injholder}, \eqref{OmegaC2a} and a parabolic regularity theorem in Hölder spaces (see \cite[Theorem 8.3.7 and Theorem 7.2.24]{WYW}) that
$\psi_{l+2}^{\varepsilon} \in C^{1+\alpha/2, 2+\alpha}(\overline{Q_T})$. Therefore, we have
\begin{equation}
\frac{1}{\varepsilon^{1/(2k+2)}} \norme{\zeta^{\varepsilon}(T,.)}_{L^{q}(\Omega)}^{1/(2k+1)} + \norme{(h^{\varepsilon})^{1/(2k+1)}}_{C^{1+\alpha/2, 2+\alpha}(\overline{Q_T})} \leq C \norme{\zeta_0}_{L^{q}(\Omega)}^{1/(2k+1)}.
\label{hxpnlinftyavecfinihold}
\end{equation}
By \eqref{hxpnlinftyavecfinihold}, we conclude the proof of \Cref{rmk2hlinfty} as in the \textbf{Step 4} by using the compact embedding \\$C^{1+\alpha/2, 2+\alpha}(\overline{Q_T}) \hookrightarrow C^{1,2}(\overline{Q_T})$ by Ascoli's theorem.
\end{rmk}
\section{Local null-controllability of general nonlinear systems}
The proof of \Cref{mainresult2} relies on the return method and \Cref{couplzonecontrole}. Thus, we just need to construct an appropriate reference trajectory $((\overline{u}, \overline{v}), \overline{h})$. The goal of this section is to prove the existence of a nontrivial trajectory of \eqref{systemef1f2} associated to $f_1$ and $f_2$ defined in \Cref{mainresult2} (see in particular \eqref{hyp1} and \eqref{hyp2}). More precisely, we have the following result.
\begin{prop}\label{existencetraj}
Let $\omega_0$ be a nonempty open subset such that $\omega_0 \subset \subset \omega$. There exist $\varepsilon > 0$, $((\overline{u},\overline{v}),\overline{h}) \in (W_T \cap L^{\infty}(Q_T))^2 \times \Big(\bigcap_{[2,+\infty)}L^{p}(Q_T)\Big)$ such that
\begin{equation*}
\left\{
\begin{array}{l l}
\partial_t \overline{u}-  \Delta \overline{u} =  f_1(\overline{u},\overline{v}) + \overline{h} 1_{\omega} &\mathrm{in}\ (0,T)\times\Omega,\\
\partial_t \overline{v} -  \Delta \overline{v} = f_2(\overline{u},\overline{v})  &\mathrm{in}\ (0,T)\times\Omega,\\
\overline{u},\overline{v}= 0&\mathrm{on}\ (0,T)\times\partial\Omega,\\
(\overline{u},\overline{v})(0,.)=(0,0),\ (\overline{u},\overline{v})(T,.)=(0,0)& \mathrm{in}\ \Omega,
\end{array}
\right.
\end{equation*}
and 
\begin{equation}\label{T/83T/8}
\forall (t,x) \in (T/8,3T/8) \times \omega_0,\ \overline{u}(t,x) \geq \varepsilon.
\end{equation}
\end{prop}
We construct the reference trajectory $((\overline{u},\overline{v}),\overline{h})$ on $(0,T/2)$ to guarantee \eqref{T/83T/8} according to the following statement.
\begin{prop}\label{defu1v1h1RM}
There exists $\varepsilon_0 > 0$ such that for every $0 < \varepsilon < \varepsilon_0$, there exists $((\overline{u_1},\overline{v_1}),\overline{h_1}) \in L^{\infty}((0,T/2)\times\Omega)^2 \times L^{\infty}((0,T/2)\times\Omega)$ satisfying 
\begin{equation}
\left\{
\begin{array}{l l}
\partial_t \overline{u_1}-  \Delta \overline{u_1} =  f_1(\overline{u_1},\overline{v_1}) + \overline{h_1} 1_{\omega} &\mathrm{in}\ (0,T/2)\times\Omega,\\
\partial_t \overline{v_1} -  \Delta \overline{v_1} = f_2(\overline{u_1},\overline{v_1})  &\mathrm{in}\ (0,T/2)\times\Omega,\\
\overline{u_1},\overline{v_1}= 0&\mathrm{on}\ (0,T/2)\times\partial\Omega,\\
(\overline{u_1},\overline{v_1})(0,.)=(0,0)& \mathrm{in}\ \Omega,
\end{array}
\right.
\label{defu1uv1T/2}
\end{equation}
and
\begin{equation} 
\forall (t,x) \in (T/8,3T/8) \times \omega_0,\ \overline{u_1}(t,x) \geq \varepsilon,
\label{propoverlineu1}
\end{equation}
\begin{equation} 
\norme{\overline{u_1}}_{L^{\infty}((0,T/2)\times\Omega)} \leq 2 \varepsilon,
\label{propoverlineu1bis}
\end{equation}
\begin{equation}
\norme{\overline{v_1}(T/2,.)}_{L^{\infty}(\Omega)} \leq C \varepsilon^{2k+1}.
\label{estvt/2}
\end{equation}
\end{prop}
\begin{proof}
Let $\varepsilon > 0$, $\overline{u_1} \in C^{\infty}(\overline{(0,T/2)\times\Omega})$ such that $supp(\overline{u_1}) \subset \subset (0,T/2) \times \omega$, \eqref{propoverlineu1} and \eqref{propoverlineu1bis} holds. By a standard Banach fixed point argument, \eqref{hyp2} and \eqref{propoverlineu1bis}, for $\varepsilon > 0$ small enough, there exists a unique solution $\overline{v_1} \in L^{\infty}((0,T/2)\times \Omega)$ of
\begin{equation}
\left\{
\begin{array}{l l}
\partial_t \overline{v_1} -  \Delta \overline{v_1} = f_2(\overline{u_1},\overline{v_1}) &\mathrm{in}\ (0,T/2)\times\Omega,\\
\overline{v_1}= 0&\mathrm{on}\ (0,T/2)\times\partial\Omega,\\
\overline{v_1}(0,.)=0& \mathrm{in}\ \Omega,
\end{array}
\right.
\label{systemevT/2NL}
\end{equation}
in the sense of \Cref{defpropsolNL}.
From \eqref{propoverlineu1bis}, \eqref{hyp2}, \eqref{systemevT/2NL} and \Cref{wpl2} (see \eqref{estlinftyfaible}), we have \eqref{estvt/2}.
Finally, we define $\overline{h_1} \in L^{\infty}((0,T/2)\times \Omega)$ thanks to the property of $supp(\overline{u_1})$ and \eqref{hyp1} (note that $f_1(0,.)=0$),
\begin{equation} \overline{h_1} := \partial_t \overline{u_1}-  \Delta \overline{u_1} -  f_1(\overline{u_1},\overline{v_1}),\label{defoverlineh1}\end{equation} 
which is supported on $(0,T/2) \times \omega$.
This ends the proof of \Cref{defu1v1h1RM}.
\end{proof}
We construct the reference trajectory $((\overline{u},\overline{v}),\overline{h})$ of \Cref{existencetraj} on $(T/2,T)$ to guarantee $(\overline{u},\overline{v})(T,.)=0$ according to the following statement, which relies on \Cref{lemcontreg} and the local invertibility of $g_1^{1/(2k+1)}$.
\begin{prop}\label{defu2v2h2RM}
Let $\varepsilon_0$ be as in \Cref{defu1v1h1RM}.\\
\indent There exists $\varepsilon_0' \in (0,\varepsilon_0)$ such that, for every $\varepsilon \in (0,\varepsilon_0')$, there exists $((\overline{u_2},\overline{v_2}),\overline{h_2}) \in L^{\infty}((T/2,T)\times\Omega)^2 \times \cap_{[2,+\infty)} L^p((T/2,T)\times\Omega)$ satisfying 
\begin{equation}
\left\{
\begin{array}{l l}
\partial_t \overline{u_2}-  \Delta \overline{u_2} =  f_1(\overline{u_2},\overline{v_2}) + \overline{h_2} 1_{\omega} &\mathrm{in}\ (T/2,T)\times\Omega,\\
\partial_t \overline{v_2} -  \Delta \overline{v_2} = f_2(\overline{u_2},\overline{v_2})  &\mathrm{in}\ (T/2,T)\times\Omega,\\
\overline{u_2},\overline{v_2}= 0&\mathrm{on}\ (0,T/2)\times\partial\Omega,\\
(\overline{u_2},\overline{v_2})(T/2,.)=(0,\overline{v_1}(T/2,.)),\quad (\overline{u_2},\overline{v_2})(T,.)=(0,0)& \mathrm{in}\ \Omega,
\end{array}
\right.
\label{defu2v2T}
\end{equation}
where $((\overline{u_1},\overline{v_1}),\overline{h_1})$ is given by \Cref{defu1v1h1RM}.
\end{prop}
\begin{proof}
We recall that $f_2(u,v) = g_1(u) g_2(v)$ (see \eqref{hyp2}).\\

\indent \textbf{Step 1:} We prove the existence of $a, \alpha, \beta > 0$ and a $C^{\infty}$-diffeomorphism $\widetilde{g_1} : (-a,a) \rightarrow(-\alpha,\beta)$ such that
\begin{equation}\label{g1g2prop}\forall x \in (-a,a),\ g_1(x):= \widetilde{g_1}(x)^{2k+1},\ g_2(x) \neq 0.\end{equation}
From \eqref{hyp2} and the Taylor formula, the map $$\widetilde{g_1}(x) := \left(\int_0^{1} \frac{(1-u)^{2k}}{(2k)!} g_1^{(2k+1)}(ux) du\right)^{1/(2k+1)} x,$$ satisfies $g_1(x) = \widetilde{g_1}(x)^{2k+1}$ for every $x \in \R$. Taking into account that $$\widetilde{g_1}'(0) = \frac{1}{(2k+1)!} g_1^{(2k+1)}(0) \neq 0\ \text{and}\ g_2(0) \neq 0\ (\text{see}\ \eqref{hyp2}),$$ there exists $a > 0$ such that $\widetilde{g_1} \in C^{\infty}((-a,a);\R)$, and $\widetilde{g_1}'(x) \neq 0$, $g_2(x) \neq 0$ for every $x \in (-a,a)$.\\
\indent We conclude from \textbf{Step 1} that $f_2(u,v) = \widetilde{g_1}^{2k+1}(u) g_2(v)$ locally around $0$.\\

\indent \textbf{Step 2:} Let  $0< \varepsilon < \varepsilon_0$ be a small parameter which will be fixed later. Let $((\overline{u_1},\overline{v_1}),\overline{h_1})$ be as in \Cref{defu1v1h1RM}.\\
\indent We apply \Cref{lemcontreg} with $(0,\tau)\leftarrow (T/2,T)$, $y_0 \leftarrow \overline{v_1}(T/2,.) \in L^{\infty}(\Omega)$. From \eqref{estvt/2}, there exists a control $H \in L^{\infty}((T/2,T)\times\Omega)$ such that 
\begin{align}
&\norme{H^{\frac{1}{2k+1}}}_{L^{\infty}((T/2,T)\times \Omega)} \leq C \norme{\overline{v_1}(T/2,.)}_{L^{\infty}(\Omega)}^{1/(2k+1)} \leq C \varepsilon, \label{estiHreg}\\
&H^{\frac{1}{2k+1}} \in  \cap_{p \in [2,+\infty)} X_{(T/2,T),p},\qquad (\text{see}\ \text{Definition}\ \ref{defiXp}),\label{Hregbis}\\
& H(T/2,.) = H(T,.) = 0\label{Hnulbis},\\
& \forall t \in [T/2,T],\ supp(H(t,.)) \subset\subset \omega\label{suppHbis},
\end{align} and the solution $\overline{v_2}$ of
\begin{equation}
\left\{
\begin{array}{l l}
\partial_t \overline{v_2} -  \Delta \overline{v_2} = H  &\mathrm{in}\ (T/2,T)\times\Omega,\\
\overline{v_2}= 0&\mathrm{on}\ (T/2,T)\times\partial\Omega,\\
\overline{v_2}(T/2,.)=\overline{v_1}(T/2,.)& \mathrm{in}\ \Omega,
\end{array}
\right.
\label{systemevTNL}
\end{equation}
satisfies 
\begin{equation}
\overline{v_2}(T,.) = 0.
\label{nuloverlinev2T}
\end{equation}
From \eqref{estiHreg}, \eqref{systemevTNL} and \Cref{wpl2}, we have
\begin{equation}
\norme{\overline{v_2}}_{L^{\infty}((T/2,T)\times\Omega)} \leq C \varepsilon^{2k+1}.
\label{estvT}
\end{equation}
Moreover, $\overline{v_2}$ is the restriction on $(T/2,T)$ of $\overline{v}$ defined by 
\begin{equation}
\left\{
\begin{array}{l l}
\partial_t \overline{v} -  \Delta \overline{v} = f_2(\overline{u_1}, \overline{v_1})1_{(0,T/2)} + H 1_{(T/2,T)}  &\mathrm{in}\ (0,T)\times\Omega,\\
\overline{v}= 0&\mathrm{on}\ (0,T)\times\partial\Omega,\\
\overline{v}(0,.)=0& \mathrm{in}\ \Omega.
\end{array}
\right.
\label{systemevTNLbis}
\end{equation}
Then, by using \Cref{defu1v1h1RM}: $(\overline{u_1}, \overline{v_1}) \in L^{\infty}(Q_T)^2$, \eqref{estiHreg}, \eqref{systemevTNLbis}, \Cref{defiXp} and \Cref{wplp}, we have
\begin{equation}
\overline{v_2} \in  \cap_{p \in [2,+\infty)} X_{(T/2,T),p}.
\label{regv2T}
\end{equation}
From \eqref{estvT}, for $\varepsilon$ sufficiently small, we have
\begin{equation}
\norme{\overline{v_2}}_{L^{\infty}((T/2,T)\times \Omega)} < a/2,
\label{estnormev2a}
\end{equation}
where $a$ is defined in \textbf{Step 1}. Therefore, from \eqref{estnormev2a} and \eqref{g1g2prop}, $g_2(\overline{v_2})^{-\frac{1}{2k+1}}$ is well-defined. Moreover, from \eqref{estiHreg}, for $\varepsilon$ sufficiently small, we have 
\begin{equation}
\label{estnormealphabeta}
\norme{H^{\frac{1}{2k+1}} g_2(\overline{v_2})^{-\frac{1}{2k+1}}}_{L^{\infty}((0,T/2)\times \Omega)} <  \max(\alpha/2,\beta/2),
\end{equation}
where $\alpha$ and $\beta$ are defined in \textbf{Step 1}.\\
\indent Then, we set \begin{align}\overline{u_2} &:= \widetilde{g_1}^{-1} \left(H^{\frac{1}{2k+1}} g_2(\overline{v_2})^{-\frac{1}{2k+1}}\right) \in L^{\infty}((T/2,T)\times \Omega),\label{defoverlineu2}\end{align} where $\widetilde{g_1}$ is defined as in \textbf{Step 1}.
From the fact that $g_2^{-\frac{1}{2k+1}} \in W^{2, \infty} (-a/2,a/2)$ (see \eqref{g1g2prop}), \eqref{estnormev2a} and \eqref{regv2T}, we check that
\begin{equation}
g_2(\overline{v_2})^{-\frac{1}{2k+1}} \in  \cap_{p \in [2,+\infty)} X_{(T/2,T),p}.
\label{g2v2reg}
\end{equation}
Taking into account that $\widetilde{g_1}^{-1} \in W^{2, \infty}(-\alpha/2, \beta/2)$, \eqref{Hregbis}, \eqref{g2v2reg} and \eqref{defoverlineu2}, we verify that 
\begin{equation}
\overline{u_2} \in  \cap_{p \in [2,+\infty)} X_{(T/2,T),p}.
\label{regu2}
\end{equation}
Finally, we define $\overline{h_2}$ thanks to \eqref{defoverlineu2} and \eqref{regu2}
\begin{equation}\overline{h_2}:= \partial_t \overline{u_2}-  \Delta \overline{u_2} -  f_1(\overline{u_2},\overline{v_2}) \in  \cap_{p \in [2,+\infty)}L^p((T/2,T)\times\Omega).,\label{defoverlineh2}\end{equation}
which is supported on $(T/2,T)\times \omega$ by \eqref{suppHbis} and \eqref{hyp1} (note that $f_1(0,.)=0$).
This ends the proof of \Cref{defu2v2h2RM}.
\end{proof}
\section{Some generalizations of the global null-controllability for “odd power systems”}\label{geneodd}

In this section, we generalize \Cref{mainresult1} to other parabolic systems. We omit the proofs because in each case, it is a slight adaptation of the strategy of \Cref{preuveglobale}.
\subsection{Linear parabolic operators} We present a natural generalization of the global null-controllability of \eqref{systemeodd} to more general \textit{linear} parabolic operators than $\partial_t - \Delta$.
\begin{prop}
Let $k \in \N^*$, $(d_1, d_2) \in (0,+\infty)^2$, $(b_1,b_2) \in (L^{\infty}(Q_T)^N)^2$, $(a_1,a_2) \in L^{\infty}(Q_T)^2$. Then, 
\begin{equation*}
\left\{
\begin{array}{l l}
\partial_t u -  d_1 \Delta u + b_1 . \nabla u + a_1 u  =   h 1_{\omega} &\mathrm{in}\ (0,T)\times\Omega,\\
\partial_t v -  d_2 \Delta v + b_2 . \nabla v + a_2 v = u^{2k+1}  &\mathrm{in}\ (0,T)\times\Omega,\\
u,v= 0&\mathrm{on}\ (0,T)\times\partial\Omega,\\
(u,v)(0,.)=(u_0,v_0)& \mathrm{in}\ \Omega,
\end{array}
\right.
\end{equation*}
is globally null-controllable.
\end{prop}
The proof is based on a Carleman estimate different from the one in \Cref{carl2lem} which can be found in \cite[Lemma 2.1]{FCG}.
\subsection{Global null-controllability result for particular superlinearities} We state a global null-controllability result linked with the global null-controllability of the \textit{semilinear heat equation}.
\begin{prop}
Let $k \in \N^*$, $f \in C^{\infty}(\R;\R)$ such that $$f(0) = 0\ \text{and}\ \frac{f(s)}{s \log^{3/2}(1 + |s|)} \rightarrow 0\ \text{when}\ |s| \rightarrow + \infty.$$ Then,
\begin{equation*}
\left\{
\begin{array}{l l}
\partial_t u -  \Delta u  =    f(u) + h 1_{\omega} &\mathrm{in}\ (0,T)\times\Omega,\\
\partial_t v -   \Delta v = u^{2k+1}  &\mathrm{in}\ (0,T)\times\Omega,\\
u,v= 0&\mathrm{on}\ (0,T)\times\partial\Omega,\\
(u,v)(0,.)=(u_0,v_0)& \mathrm{in}\ \Omega,
\end{array}
\right.
\end{equation*}
is globally null-controllable.
\end{prop}
The proof is based on the global null-controllability of the semilinear heat equation with superlinearity as the function $f$. Particularly, we can see \cite[Theorem 1.2]{FCZ} proved by Enrique Fernandez-Cara and Enrique Zuazua or \cite[Theorem 1.7]{FCG}. 
\subsection{Global null-controllability for all “power systems”}
Let $n \in \N^*$. We have seen that \eqref{systemeu^n}, with $n$ an even integer, is not (globally) null-controllable by the maximum principle (see \Cref{resultatnegatifeven}) but \eqref{systemeu^n}, with $n$ an odd integer, is globally null-controllable (see \Cref{mainresult1}). In this section, we consider the following system:
\begin{equation}
\left\{
\begin{array}{l l}
\partial_t u -  \Delta u =   h 1_{\omega} &\mathrm{in}\ (0,T)\times\Omega,\\
\partial_t v -  \Delta v = |u|^{n-1}u  &\mathrm{in}\ (0,T)\times\Omega,\\
u,v= 0&\mathrm{on}\ (0,T)\times\partial\Omega,\\
(u,v)(0,.)=(u_0,v_0)& \mathrm{in}\ \Omega.
\end{array}
\right.
\tag{PowerO}
\label{systemeu^nO}
\end{equation}
\begin{defi}\label{defracinepuissimpgene}
The mapping $\Phi_{n}: x \in \R \mapsto |x|^{n-1}x \in \R$ is one-to-one. We note its inverse function $\Phi_{n}^{-1}$.
\end{defi}
\begin{rmk}
For $n$ an even integer, $\Phi_{n}(u) = u^{n}$ if $u \geq 0$ and $\Phi_{n}(u) = -u^{n}$ if $u < 0$. Whereas for $n$ an odd power, $\Phi_{n}(u)  = u^{n}$ for every $u \in \R$. 
\end{rmk}
We have a generalization of \Cref{mainresult1}.
\begin{theo}\label{mainresult1bis}
The system \eqref{systemeu^nO} is globally null-controllable (in the sense of \Cref{deflocalglobalcontr}).\\
\indent More precisely, there exists $(C_p)_{p \in [2,+\infty)} \in (0,\infty)^{[2,+\infty)}$ such that for every initial data $(u_0,v_0) \in L^{\infty}(\Omega)^2$, there exists a control $h \in \bigcap\limits_{p \in [2, +\infty)}L^p(Q_T)$ satisfying
\begin{equation}\label{estimcontrolmr1bis}
\forall p \in [2, +\infty),\ \norme{h}_{L^p(Q_T)} \leq C_p \left(\norme{u_0}_{L^{\infty}(\Omega)} + \norme{v_0}_{L^{\infty}(\Omega)}^{1/n}\right),
\end{equation}
and the solution $(u,v)$ of \eqref{systemeu^nO} verifies
$$(u,v)(T,.)=(0,0).$$
\end{theo}
The strategy of proof of \Cref{mainresult1bis} is the same as for \Cref{mainresult1}. It is based on the following \textit{key} result.
\begin{prop}\label{lemcontreggene}
For every $\tau > 0$, there exists $C_{\tau} >0$ such that \\for every $y_0 \in L^{\infty}(\Omega)$, there exists a control $h_{\tau} \in L^{\infty}(Q_{\tau})$ which verifies 
\begin{align*}
& \norme{\Phi_{n}^{-1}(h_{\tau})}_{L^{\infty}(Q_{\tau})} \leq C_{\tau} \norme{y_0}_{L^{\infty}(\Omega)}^{1/n},\\
&\Phi_{n}^{-1}(h_{\tau}) \in \bigcap\limits_{p \in [2, +\infty)} X_{\tau,p},\\
& \forall p \in [2,+\infty),\ \exists C_{\tau,p} >0,\  \norme{\Phi_{n}^{-1}(h_{\tau})}_{X_{\tau,p}} \leq C_{\tau,p} \norme{y_0}_{L^{\infty}(\Omega)}^{1/n},\\ 
&h_{\tau}(0,.)=h_{\tau}(\tau,.)=0,\\
&\forall t \in [0,\tau],\ supp(h_{\tau}(t,.)) \subset\subset \omega,
\end{align*} 
such that the solution $y \in  L^{\infty}(Q_{\tau})$ of \eqref{heat} satisfies $y(\tau,.)=0$.
\end{prop}
The proof of \Cref{lemcontreggene} is a slight adaptation of \Cref{existencecontrolodd}. First, we get a Carleman estimate in $L^{n+1}$ (see \Cref{carl2+2cor}). Secondly, we use a penalized duality method in $L^{(n+1)/n}$ as in \Cref{RUM} taking into account that the antiderivative of the continuous mapping $\Phi_{n}^{-1}$ (see \Cref{defracinepuissimpgene}) is the strictly convex function \[x \in \R \mapsto \frac{n}{n+1}\abs{x}^{\frac{n+1}{n}} := \left\{
\begin{array}{c l}
\frac{n}{n+1}\exp(\frac{n+1}{n} \log (\abs{x})) & \mathrm{if}\  x \neq 0,\\
0 & \mathrm{if}\  x =0.
\end{array}
\right.\]
\subsection{Global null-controllability for “even power systems” in $\C$}
Let $k \in \N^*$. We have seen in \Cref{resultatnegatifeven} that global null-controllability does not hold for 
\begin{equation*}
\left\{
\begin{array}{l l}
\partial_t u -  \Delta u =   h 1_{\omega} &\mathrm{in}\ (0,T)\times\Omega,\\
\partial_t v -  \Delta v = u^{2k}  &\mathrm{in}\ (0,T)\times\Omega,\\
u,v= 0&\mathrm{on}\ (0,T)\times\partial\Omega,\\
(u,v)(0,.)=(u_0,v_0)& \mathrm{in}\ \Omega.
\end{array}
\right.
\tag{Even}
\label{systemeeven}
\end{equation*}
A natural question, asked by Luc Robbiano, is: what happens if we consider \textit{complex-valued} functions? A positive answer i.e. a global null-controllability result for \eqref{systemeeven} with $k=1$ is given in \cite{CGR} (see \cite[Theorem 3]{CGR}). Here, we want to generalize this result for every $k \in \N^*$. We have the following result.
\begin{theo}\label{theoevenpower}
The system \eqref{systemeeven} is globally null-controllable (in the sense of \Cref{deflocalglobalcontr}).\\
\indent More precisely, there exists $(C_p)_{p \in [2,+\infty)} \in (0,\infty)^{[2,+\infty)}$ such that for every initial data $(u_0,v_0) \in L^{\infty}(\Omega)^2$, there exists a control $h \in \bigcap\limits_{p \in [2, +\infty)}L^p(Q_T;\C)$ satisfying
\begin{equation}\label{estimcontrolmr1even}
\forall p \in [2, +\infty),\ \norme{h}_{L^p(Q_T;\C)} \leq C_p \left(\norme{u_0}_{L^{\infty}(\Omega)} + \norme{v_0}_{L^{\infty}(\Omega)}^{1/(2k)}\right),
\end{equation}
and the solution $(u,v) \in L^{\infty}(Q_T,\C)^2$ of \eqref{systemeeven} verifies
$$(u,v)(T,.)=(0,0).$$
\end{theo}
The strategy of proof of \Cref{theoevenpower} is the same as for \Cref{mainresult1} (see \Cref{preuveglobale}). The first step of the proof i.e. \Cref{fsetape} does not change but we have to modify some arguments given in \Cref{scetape}.\\
\indent Let us fix $(u_0,v_0)\in L^{\infty}(\Omega)^2$ until the end of the section.
\subsubsection{First step of the proof: steer $u$ to $0$}
First, we find a control of \eqref{systemeeven} which steers $u$ to $0$ in time $T/2$ (see the proof of \Cref{controlT/2}). 
\begin{prop}\label{controlevenT/2}
There exists $h_1 \in L^{\infty}((0,T/2)\times \Omega;\R)$ satisfying 
\begin{equation}
\norme{h_1}_{L^{\infty}((0,T/2)\times \Omega;\R)} \leq C \norme{u_0}_{L^{\infty}(\Omega)},
\label{esticontrolT/2even}
\end{equation}
such that the solution $(u_1,v_1) \in L^{\infty}((0,T/2)\times \Omega;\R)^2 $ of
\begin{equation*}
\left\{
\begin{array}{l l}
\partial_t u_1 -  \Delta u_1 =   h_1 1_{\omega} &\mathrm{in}\ (0,T/2)\times\Omega,\\
\partial_t v_1 -  \Delta v_1 = u_1^{2k}  &\mathrm{in}\ (0,T/2)\times\Omega,\\
u_1,v_1= 0&\mathrm{on}\ (0,T/2)\times\partial\Omega,\\
(u_1,v_1)(0,.)=(u_0,v_0)& \mathrm{in}\ \Omega,
\end{array}
\right.
\label{systemeoddT/2even}
\end{equation*}
satisfies $u_1(T/2,.)=0$. Moreover, we have 
\begin{equation}\label{estiv1T/2even}
\norme{v_1(T/2,.)}_{L^{\infty}(\Omega)} \leq C \left(\norme{u_0}_{L^{\infty}(\Omega)}^{2k} + \norme{v_0}_{L^{\infty}(\Omega)}\right).
\end{equation}
\end{prop}
\subsubsection{Second step of the proof: steer $v$ to $0$}
The aim of this part is to find a \textit{complex} control of \eqref{systemeeven} which steers $v$ to $0$ (and $u$ from $0$ to $0$) in time $T$.
\begin{prop}\label{controlTeven}
Let $((u_1,v_1),h_1)$ as in \Cref{controlevenT/2}.\\
\indent There exists a control $h_2 \in \bigcap_{p \in [2, +\infty)}L^p((T/2,T)\times \Omega;\C)$ satisfying
\begin{equation}
\forall p \in [2,+\infty),\ \norme{h_2}_{L^{p}((T/2,T)\times \Omega);\C)} \leq C_p\left( \norme{u_0}_{L^{\infty}(\Omega)} + \norme{v_0}_{L^{\infty}(\Omega)}^{1/(2k)}\right).
\label{esticontrolTeven}
\end{equation}
such that the solution $(u_2,v_2) \in L^{\infty}((T/2,T)\times \Omega;\C)^2 $ of
\begin{equation*}
\left\{
\begin{array}{l l}
\partial_t u_2 -  \Delta u_2 =   h_2 1_{\omega} &\mathrm{in}\ (T/2,T)\times\Omega,\\
\partial_t v_2 -  \Delta v_2 = u_2^{2k}  &\mathrm{in}\ (T/2,T)\times\Omega,\\
u_2,v_2= 0&\mathrm{on}\ (T/2,T)\times\partial\Omega,\\
(u_2,v_2)(T/2,.)=(0,v_1(T/2,.))& \mathrm{in}\ \Omega,
\end{array}
\right.
\label{systemeevenT}
\end{equation*}
satisfies $(u_2,v_2)(T,.)=(0,0)$.
\end{prop}
\indent Our approach consists in looking at the second equation of \eqref{systemeodd} like a controlled heat equation where the \textit{state} is $v(t,.)$ and the \textit{control input} is $u^{2k}(t,.)$. Here, \textbf{the question consists in proving that the heat equation is null-controllable with a localized control which \textit{is as an even power}} of a regular \textit{complex} function.

Now, we can prove \Cref{controlTeven}.
\begin{proof}
We apply \Cref{lemcontreggene} with $n=4k$, $(0,\tau)\leftarrow(T/2,T)$, $y_0 \leftarrow v_1(T/2,.) \in L^{\infty}(\Omega)$. Then, there exists a control $H\in L^{\infty}((T/2,T)\times\Omega)$ such that 
\begin{align}
&\Phi_{4k}^{-1}(H) \in \bigcap_{p \in [2, +\infty)}X_{(T/2,T),p}\label{Hregeven}\\
&\forall p \in [2,+\infty),\ \norme{\Phi_{4k}^{-1}(H)}_{X_{(T/2,T),p}} \leq C_{p} \norme{v_1(T/2,.)}_{L^{\infty}(\Omega)}^{1/(4k)},\label{Hregestieven}\\
& H(T/2,.) = H(T,.) = 0\label{Hnuleven},\\
& \forall t \in [T/2,T],\ supp(H(t,.)) \subset\subset \omega\label{suppHeven},
\end{align}
and the solution $v_2$ of 
\begin{equation}
\left\{
\begin{array}{l l}
\partial_t v_2 -  \Delta v_2 = H  &\mathrm{in}\ (T/2,T)\times\Omega,\\
v_2= 0&\mathrm{on}\ (T/2,T)\times\partial\Omega,\\
v_2(T/2,.)=v_1(T/2,.)& \mathrm{in}\ \Omega,
\end{array}
\right.
\label{systemevTeven}
\end{equation}
satisfies \begin{equation}v_2(T,.) = 0.\label{nulvTeven}\end{equation}\\
We introduce the notation $$\alpha := e^{\frac{i \pi}{2k}} \in \C.$$
We take $u_2$, the \textit{complex-valued function}, as
\begin{equation}
u_2 := \Big(\left((\Phi_{4k}^{-1}(H)\right)^{+}\Big)^2 + \alpha \Big(\left((\Phi_{4k}^{-1}(H)\right)^{-}\Big)^2, \label{defu2even}\end{equation}
where the positive and negative parts of a real number $x$ are defined as follows \[x^{+} := \max(x,0),\qquad x^{-} := -\min(x,0).\]
From \Cref{defracinepuissimpgene} and \eqref{defu2even}, we verify that
\begin{equation}u^{2k} = \Big(\left((\Phi_{4k}^{-1}(H)\right)^{+}\Big)^{4k} + \alpha^{2k}  \Big(\left((\Phi_{4k}^{-1}(H)\right)^{-}\Big)^{4k} = H^{+} - H^{-} =  H.\label{u22k=H}\end{equation}
Moreover, we have
\begin{equation}
x \mapsto (x^{+})^{2} \in W_{loc}^{2, \infty}(\R),\qquad x \mapsto (x^{-})^{2} \in W_{loc}^{2, \infty}(\R).
\label{regpartpospartnet}
\end{equation}
From \eqref{Hregeven}, \eqref{defu2even} and \eqref{regpartpospartnet}, we have
\begin{equation}
u_2  \in \left(\bigcap_{p \in [2, +\infty)}X_{(T/2,T),p}\right)\cap L^{\infty}((T/2,T)\times\Omega;\C).\label{u2regeven}\end{equation}
We have, from \eqref{Hnuleven},
\begin{equation}u_2(T/2,.)= u_2(T,.) = 0.\label{u2nulextreven}\end{equation}
Then, we set, from \eqref{defu2even}, \eqref{u2regeven} and \eqref{suppHeven},
\begin{equation}h_2:= \partial_t u_2 - \Delta u_2 \in \bigcap\limits_{p \in [2,+\infty)}L^p((T/2,T)\times\Omega;\C),\label{defh2even}\end{equation}
which is supported in $(T/2,T)\times\omega$.
From \eqref{Hregestieven}, \eqref{estiv1T/2even} and \eqref{defu2even}, we have 
\begin{equation}
\forall p \in [2,+\infty),\ \norme{h_2}_{L^{p}((T/2,T)\times \Omega);\C)} \leq C_p\left( \norme{u_0}_{L^{\infty}(\Omega)} + \norme{v_0}_{L^{\infty}(\Omega)}^{1/(2k)}\right).
\label{esticontrolTevenbis}
\end{equation}
By using \eqref{suppHeven}, \eqref{systemevTeven}, \eqref{nulvTeven}, \eqref{u22k=H}, \eqref{u2regeven}, \eqref{u2nulextreven}, \eqref{defh2even} and \eqref{esticontrolTevenbis}, we check that $((u_2,v_2),h_2)$ satisfies \Cref{controlTeven}.
\end{proof}
\appendix
\section{}\label{appendix}

\DeactivateToc
\subsection{Proof of the uniqueness of the point $2.$ of \Cref{defpropsolNL}}\label{proofuniqueness}

\begin{proof}
Let $(u_0,v_0) \in L^{\infty}(\Omega)^2$, $h \in L^2(Q_T)$.\\
\indent Let $(u,v) \in (W_T \cap L^{\infty}(Q_T))^2$ and $(\widetilde{u},\widetilde{v}) \in (W_T\cap L^{\infty}(Q_T))^2$ be two solutions of \eqref{systemef1f2}. Then, the function $(\widehat{u},\widehat{v}):=(u-\widetilde{u},v - \widetilde{v}) \in (W_T \cap L^{\infty}(Q_T))^2$ satisfies (in the weak sense)
\begin{equation}
\left\{
\begin{array}{l l}
\partial_t \widehat{u} -  \Delta \widehat{u} =   f_1(u,v) - f_1(\widetilde{u},\widetilde{v}) &\mathrm{in}\ (0,T)\times\Omega,\\
\partial_t \widehat{v} -  \Delta \widehat{v} =f_2(u,v) - f_2(\widetilde{u},\widetilde{v}) &\mathrm{in}\ (0,T)\times\Omega,\\
\widehat{u},\widehat{v}= 0&\mathrm{on}\ (0,T)\times\partial\Omega,\\
(\widehat{u},\widehat{v})(0,.)=(0,0)& \mathrm{in}\ \Omega.
\end{array}
\right.
\label{systemediffsol}
\end{equation}
By taking $(w_1,w_2):=(\widehat{u},\widehat{v})$ in the variational formulation of \eqref{systemediffsol} (see \eqref{formvarNL1} and \eqref{formvarNL2}) and by using the fact that the mapping $t \mapsto \norme{(\widehat{u}(t),\widehat{v}(t))^{T}}_{L^2(\Omega)^2}^2$ is absolutely continuous (see \cite[Section 5.9.2, Theorem 3]{E}) with for a.e. $0 \leq t \leq T$,
\[\frac{d}{dt} \norme{(\widehat{u}(t),\widehat{v}(t))^{T}}_{L^2(\Omega)^2}^2 = 2 \left(\Big((\partial_t \widehat{u}(t),\widehat{u}(t)),(\partial_t \widehat{v}(t),\widehat{v}(t))\Big)^{T}\right)_{(H^{-1}(\Omega)^2,H_0^1(\Omega)^2)},\]we find that for a.e. $0 \leq t \leq T$,
\begin{align*}
&\frac{1}{2}\frac{d}{dt} \left(\norme{(\widehat{u}(t),\widehat{v}(t))^T}_{L^2(\Omega)^2}^2\right) + \norme{(\nabla \widehat{u},\nabla \widehat{v})^{T}}_{L^2(\Omega)^2}^2\\
& = \left(\Big((f_1(u,v) - f_1(\widetilde{u},\widetilde{v}),\widehat{u}), (f_2(u,v) - f_2(\widetilde{u},\widetilde{v}),\widehat{v})\Big)^{T}\right)_{(L^2(\Omega)^2,L^2(\Omega)^2)}.
\label{ddtv}
\end{align*}
By using the facts that $(u,\widetilde{u})\in L^{\infty}(Q_T)^2$ and $(x,y) \mapsto f_1(x,y)$, $(x,y) \mapsto f_2(x,y)$ are locally Lipschitz on $\R^2$, we find the differential inequality
\begin{equation*}
\frac{d}{dt} \left(\norme{(\widehat{u}(t),\widehat{v}(t))^T}_{L^2(\Omega)^2}^2\right) \leq C\left(\norme{(\widehat{u}(t),\widehat{v}(t))^T}_{L^2(\Omega)^2}^2\right)\  \text{for\ a.e.}\ 0 \leq t \leq T.
\label{ddtvbis}
\end{equation*}
Gronwall lemma and the initial condition $(\widehat{u}(0),\widehat{v}(0)) = (0,0)$ (see \eqref{initNL}) prove that $(\widehat{u}(t),\widehat{v}(t))=(0,0)$. Consequently, $(u,v) =(\widetilde{u},\widetilde{v})$.
\end{proof}

\subsection{Proof of the general case for \Cref{carl4usefaiblecor}}\label{appendixdensity}
\begin{proof}
Let $\varphi_T \in L^{2k+2}(\Omega)$ and $(\varphi_{T,n})_{n \in \N} \in (C_0^{\infty}(\Omega))^{\N}$ such that
\begin{equation}
\varphi_{T,n} \underset{n \rightarrow + \infty} \rightarrow \varphi_T \ \text{in}\ L^{2k+2}(\Omega).
\label{convinit}
\end{equation}
We denote by $(\varphi_n)_{n \in \N}$ the sequence of solutions of  
\begin{equation}
\left\{
\begin{array}{l l}
-\partial_t \varphi_n - \Delta\varphi_n = 0 &(0,T)\times\Omega,\\
\varphi_n = 0 &(0,T)\times\partial\Omega,\\
\varphi_n(T,.)=\varphi_{T,n} &\Omega.
\end{array}
\right.
\label{equationphiadjn}
\end{equation}
The estimates \eqref{carl2k+2usefaible} and \eqref{carl2k+2usefortpart} hold true for $(\varphi_n)_{n \in \N}$ by the \textbf{Step 1} of the proof of \Cref{carl4usefaiblecor}. Moreover, from \eqref{convinit}, \eqref{equationphiadjn}, \eqref{equationphiadj} and \Cref{wpl2} (particularly \eqref{estlinftyfaible} for $p=2k+2$), we have
\begin{align}
\norme{\varphi_n - \varphi}_{L^{2k+2}(Q_T)} \leq C \norme{\varphi_{T,n} - \varphi_T}_{L^{2k+2}(\Omega)} \underset{n \rightarrow + \infty}\rightarrow 0,
\label{convsol}
\end{align}
where $\varphi \in L^{2k+2}(Q_T)$ is the solution of \eqref{equationphiadj}. By using
\begin{equation}
\chi^{2k+2} e^{-(k+1)s \rho(x) \eta(t)}(s\eta)^{3(k+1)} \in L^{\infty}(Q_T),
\label{estlinfty}
\end{equation}
and \eqref{convsol}, we get
\begin{align}\label{convobs}
&\int\int_{(0,T)\times\omega} \chi^{2k+2} e^{-(k+1)s \rho \eta}(s\eta)^{3(k+1)}|\varphi_n|^{2k+2}\\
& \underset{n \rightarrow + \infty} \rightarrow \int\int_{(0,T)\times\omega} \chi^{2k+2} e^{-(k+1)s \rho \eta}(s\eta)^{3(k+1)}|\varphi|^{2k+2}.\notag
\end{align}
From \eqref{carl2k+2usefaible}, \eqref{carl2k+2usefortpart} applied to $(\varphi_n)_{n \in \N}$ and \eqref{convobs}, we deduce that $(\varphi_n(0,.))_{n \in \N}$, (respectively $( e^{-\rho s \eta /2} (s \eta)^{-m/2} \varphi_n)_{n \in \N}$, respectively $( e^{-\rho s \eta /2} (s \eta)^{-(m+2)/2} \nabla\varphi_n)_{n \in \N}$) is bounded in $L^{2k+2}(\Omega)$, (respectively $L^{2k+2}(Q_T)$, respectively $L^{2k+2}(Q_T)$) which is a Banach reflexive space. Then, up to a subsequence, we can assume that 
\begin{equation}
\varphi_n(0,.) \underset{n \rightarrow + \infty}\rightharpoonup \varphi(0,.)\ \text{in}\ L^{2k+2}(\Omega),
\label{convweak1}
\end{equation}
\begin{equation}
e^{-\rho s \eta /2} (s \eta)^{-m/2} \varphi_n \underset{n \rightarrow + \infty}\rightharpoonup e^{-\rho s \eta /2} (s \eta)^{-m/2} \varphi\ \text{in}\ L^{2k+2}(Q_T),
\label{convweak2}
\end{equation}
\begin{equation}
e^{-\rho s \eta /2} (s \eta)^{-(m+2)/2}\nabla\varphi_n \underset{n \rightarrow + \infty}\rightharpoonup e^{-\rho s \eta /2} (s \eta)^{-(m+2)/2} \nabla\varphi\ \text{in}\ L^{2k+2}(Q_T).
\label{convweak3}
\end{equation}
In particular, we have
\begin{align}
\label{proofcarl4faible}
\norme{\varphi(0,.)}_{L^{2k+2}(\Omega)}^{2k+2}
&\leq \liminf_{n \rightarrow + \infty} \norme{\varphi_n(0,.)}_{L^{2k+2}(\Omega)}^{2k+2}\\
 &\leq C_s \liminf_{n \rightarrow + \infty}\int\int_{(0,T)\times\omega} \chi^{2k+2} e^{-(k+1)s \rho(x) \eta(t)}(s\eta)^{3(k+1)}|\varphi_n|^{2k+2} dxdt \notag\\
&\leq C_s \int\int_{(0,T)\times\omega} \chi^{2k+2} e^{-(k+1)s \rho(x) \eta(t)}(s\eta)^{3(k+1)}|\varphi|^{2k+2} dxdt\notag,
\end{align}
and
\begin{align}
\label{proofcarl4fort}
&\int\int_{(0,T)\times\Omega} e^{-(k+1)s \rho \eta}((s\eta)^{-(k+1)m} |\varphi|^{2k+2} + (s \eta)^{-(k+1)(m+2)} |\nabla \varphi |^{2k+2})\\
&\leq \liminf_{n \rightarrow + \infty}  \int\int_{(0,T)\times\Omega} e^{-(k+1)s \rho \eta}((s\eta)^{-(k+1)m} |\varphi_n|^{2k+2} + (s \eta)^{-(k+1)(m+2)} |\nabla \varphi_n |^{2k+2})\notag\\
&\leq C \liminf_{n \rightarrow + \infty}\int\int_{(0,T)\times\omega} \chi^{2k+2} e^{-(k+1)s \rho \eta}(s\eta)^{3(k+1)}|\varphi_n|^{2k+2} \notag\\
&\leq C \int\int_{(0,T)\times\omega} \chi^{2k+2} e^{-(k+1)s \rho \eta}(s\eta)^{3(k+1)}|\varphi|^{2k+2}.\notag
\end{align}
The estimates \eqref{proofcarl4faible} and \eqref{proofcarl4fort} conclude the proof.
\end{proof}
\bibliographystyle{plain}
\small{\bibliography{bibliordnonlin}}

\end{document}